\documentclass[final,12pt]{amsart}
\usepackage{tikz-cd}
\usepackage[colorlinks=false]{hyperref}
\usepackage{amsmath, amsbsy, amssymb, amsthm}
\usepackage[top=3cm,bottom=3cm,left=2.5cm,right=2.5cm]{geometry}
\usepackage{enumitem}
\usepackage{graphicx}
\usepackage{bm}
\usepackage{upgreek}
\usepackage{mathrsfs}
\usepackage{mathtools}
\usepackage{stmaryrd}
\usepackage[scr=euler]{mathalfa}
\usepackage{multicol}
\usepackage[UKenglish]{babel}
\usepackage{proof}

\providecommand{\noopsort}[1]{}

\raggedcolumns

\theoremstyle{plain}
\newtheorem{theorem}{Theorem}[section]
\newtheorem{proposition}[theorem]{Proposition}
\newtheorem{lemma}[theorem]{Lemma}
\newtheorem{corollary}[theorem]{Corollary}

\newtheorem*{claim*}{Claim}

\theoremstyle{definition}
\newtheorem{definition}[theorem]{Definition}

\newtheorem{remark}[theorem]{Remark}
\newtheorem*{notation}{Notation}

\newtheorem*{relatedwork}{Related work}
\newtheorem*{ack}{Acknowledgements}

\makeatletter
\newlength\xvec@height%
\newlength\xvec@depth%
\newlength\xvec@width%
\newcommand{\xvec}[2][]{%
  \ifmmode%
    \settoheight{\xvec@height}{$#2$}%
    \settodepth{\xvec@depth}{$#2$}%
    \settowidth{\xvec@width}{$#2$}%
  \else%
    \settoheight{\xvec@height}{#2}%
    \settodepth{\xvec@depth}{#2}%
    \settowidth{\xvec@width}{#2}%
  \fi%
  \def\xvec@arg{#1}%
  \def\xvec@dd{:}%
  \def\xvec@d{.}%
  \raisebox{.2ex}{\raisebox{\xvec@height}{\rlap{%
    \kern.05em
    \begin{tikzpicture}[scale=1]
    \pgfsetroundcap
    \draw (.05em,0)--(\xvec@width-.05em,0);
    \draw (\xvec@width-.05em,0)--(\xvec@width-.15em, .075em);
    \draw (\xvec@width-.05em,0)--(\xvec@width-.15em,-.075em);
    \ifx\xvec@arg\xvec@d%
      \fill(\xvec@width*.45,.5ex) circle (.5pt);%
    \else\ifx\xvec@arg\xvec@dd%
      \fill(\xvec@width*.30,.5ex) circle (.5pt);%
      \fill(\xvec@width*.65,.5ex) circle (.5pt);%
    \fi\fi%
    \end{tikzpicture}%
  }}}%
  #2%
}
\makeatother

\renewcommand{\vec}[1]{\xvec[]{#1}}

\newcommand{\op}{\mathrm{op}} 
\DeclareMathOperator{\V}{\mathbb{V}} 
\DeclareMathOperator{\Max}{Max} 
\DeclareMathOperator{\C}{C}

\newcommand{\m}{\mathfrak{m}} 

\newcommand{\f}{\mathfrak{f}} 
\newcommand{\g}{\mathfrak{g}} 
\newcommand{\h}{\mathfrak{h}} 

\newcommand{\id}[1]{{\rm Id}_{#1}}
\DeclareMathOperator{\ev}{ev} 
\renewcommand{\theta}{\vartheta}
\renewcommand{\phi}{\varphi}
\newcommand{\w}{\widehat} 
\newcommand{\li}{\widetilde} 
\renewcommand{\tilde}{\widetilde}
\renewcommand{\o}{\overline}
\renewcommand{\u}{\underline}
\newcommand{\Lan}{\mathbb{L}} 
\newcommand{\Log}{\vDash_{\Delta}} 
\newcommand{\Lo}{\vdash_{\Delta}} 
\newcommand{\var}{\mathbf{Var}}

\DeclareMathOperator{\ha}{\triangleleft} 

\DeclareMathOperator{\fm}{\mathnormal{f}_{\frac{1}{2}}}
\renewcommand{\d}{\ensuremath{\textrm d}} 

\newcommand{\x}{\overline{x}}
\newcommand{\y}{\overline{y}}
\newcommand{\z}{\overline{z}}
\newcommand{\T}{\mathcal{T}} 
\newcommand{\F}{\mathcal{F}} 
\newcommand{\LT}{\mathcal{LT}} 
\newcommand{\A}{A} 
\newcommand{\B}{B} 
\newcommand{\ep}{\varepsilon}
\newcommand{\Si}{\Sigma}
\newcommand{\Ga}{\Gamma}
\renewcommand{\Form}{\mathbf{Form}}
\DeclareMathOperator{\card}{\mathrm{Card}}
\DeclareMathOperator{\Maxi}{\mathfrak{M}} 
\renewcommand{\k}{\Bbbk} 

\newcommand{\Luk}{\L_{\infty}} 

\DeclareMathOperator{\Set}{\mathscr{SET}}
\DeclareMathOperator{\KH}{\mathscr{KH}} 
\DeclareMathOperator{\lGr}{\mathscr{LA}}
\DeclareMathOperator{\MV}{\mathscr{MV}}
\renewcommand{\int}{[0,1]}
\newcommand{\R}{\mathbb{R}}
\newcommand{\N}{\mathbb{N}}

\newcommand{\epi}{\twoheadrightarrow}
\newcommand{\mono}{\hookrightarrow}
\newcommand{\quot}{\uprho}

\renewcommand{\leq}{\leqslant}

\newcommand{\Th}[1]{\Sigma\llbracket #1 \rrbracket} 

\title{Beth definability and the Stone-Weierstrass Theorem}
\thanks{This project has received funding from the European Union's Horizon 2020 research and innovation programme under the Marie Sk{\l}odowska-Curie grant agreement No 837724, and from the grant GA17-04630S of the Czech Science Foundation.}
\thanks{\emph{To appear in the Annals of Pure and Applied Logic.}}

\author{Luca Reggio}
\address{Department of Computer Science, University of Oxford, UK}
\email{luca.reggio@cs.ox.ac.uk}

\begin{document}

\maketitle

\begin{abstract}
The Stone-Weierstrass Theorem for compact Hausdorff spaces is a basic result of functional analysis with far-reaching consequences. We introduce an equational logic $\Log$ associated with an infinitary variety $\Delta$ and show that the Stone-Weierstrass Theorem is a consequence of the Beth definability property of $\Log$, stating that every implicit definition can be made explicit. Further, we define an infinitary propositional logic $\Lo$ by means of a Hilbert-style calculus and prove a strong completeness result whereby the semantic notion of consequence associated with $\Lo$ coincides with $\Log$.
\end{abstract}

\section{Introduction}
Weierstrass' Approximation Theorem states that any continuous real-valued function defined on a closed real interval can be uniformly approximated by polynomials. In 1937, Marshall Stone proved a vast generalisation of this theorem~\cite{st2}; nowadays known as the Stone-Weierstrass Theorem for compact Hausdorff spaces, this is a fundamental result of functional analysis with far-reaching consequences.

Let $X$ be a non-empty compact Hausdorff space and let $\C(X,\R)$ be the collection of all continuous functions $X\to \R$, where we denote by $\R$ the set of real numbers equipped with the usual Euclidean topology. 
The set $\C(X,\R)$ is equipped with the \emph{uniform metric} $\varrho\colon \C(X,\R)\times\C(X,\R)\to [0,\infty)$ given by
\[
\forall \f,\g\in\C(X,\R), \ \ \varrho(\f,\g)\coloneqq \sup_{x\in X}{|\f(x)-\g(x)|}.
\]
The Stone-Weierstrass Theorem provides sufficient conditions for a subset $G\subseteq \C(X,\R)$ to be dense in the topology induced by the uniform metric. Recall that a subset $G$ of $\C(X,\R)$ is said to \emph{separate the points of $X$} if, for any two distinct points $x,y\in X$, there exists $\f\in G$ such that $\f(x)\neq \f(y)$. The Stone-Weierstrass Theorem can be phrased as follows:
\begin{theorem}[{cf.~\cite[Theorem~82]{st2}}]\label{t:SW-lattice-version}
Let $X$ be a non-empty compact Hausdorff space and let $G$ be a subset of $\C(X,\R)$ satisfying the following properties:
\begin{enumerate}[label=\textup{(\roman*)}]
\item $G$ separates the points of $X$;
\item $G$ contains the constant function of value $1$;
\item $\f\in G$ and $r\in\R$ imply $r \cdot \f\in G$;
\item if $\f,\g\in G$, then $\f+\g\in G$ and $\max{\{\f,\g\}}\in G$.
\end{enumerate}
Then $G$ is dense in $\C(X,\R)$ in the topology induced by the uniform metric.
\end{theorem}
\begin{remark}
For a nice exposition of the previous result, along with a proof relying on the closure of the set $G$ under lattice-theoretic operations, see~\cite{StoneWe1948}. The version stated above, involving the operations $+$ and $\max$, seems to be more widespread in analysis (see, e.g., \cite[Theorem~7.29]{HR65}).
\end{remark}

On the other hand, Beth definability is a strong property that a logic may, or may not, satisfy. Informally, it states that every property that can be defined implicitly admits an explicit definition. Beth's theorem, proved in 1953, states that first-order logic has the Beth definability property~\cite{Beth1953}. This result sheds light on a phenomenon that occurs frequently in the mathematical practice. For instance, a real closed field $F$ (i.e., a field elementarily equivalent to $\R$) admits a unique total order that turns it into an ordered field. This order is given by the first-order formula
\[
\forall x,y\in F \ (x\leq y \, \leftrightarrow \, \exists z \, (y-x= z^2)).
\]
Beth's theorem tells us that this is no coincidence: whenever a property can be described in a unique way, there is a first-order formula explicitly defining it.
In the context of algebraic logic, the Beth definability property has been extensively investigated, cf.~\cite[\S 5.6]{HMT1985} or \cite{Hoogland2000,BH2006}. For a precise definition in our setting, see Section~\ref{s:definability}.

The aim of this paper is to show that the Stone-Weierstrass Theorem for compact Hausdorff spaces can be seen as a consequence of the Beth definability property for a certain equational logic $\Log$, thus obtaining a ``logical proof'' of the Stone-Weierstrass Theorem.
(In fact, the Beth definability property for $\Log$ can be regarded as an equivalent form of the Stone-Weierstrass Theorem, cf.\ Remark~\ref{rem:SW-implies-Beth}.)
The logic $\Log$ is the equational consequence relation associated with the variety $\Delta$ of $\delta$-algebras introduced in~\cite{MR2017}. The latter is not a variety in the usual sense, \emph{\`{a} la} Birkhoff, in that its signature contains an operation symbol of countably infinite arity. Each member of $\Delta$ has a reduct of MV-algebra, the algebraic counterpart to {\L}ukasiewicz infinite-valued propositional logic $\Luk$~\cite{Luk1920,LT1930}, and so $\Log$ is an infinitary extension of $\Luk$ (this extension is, in fact, conservative: every equation between MV-algebraic terms that holds in all $\delta$-algebras must hold in all MV-algebras). 

\vspace{0.5em}
This paper is organised as follows. In Section~\ref{s:prel} we provide the necessary background on Abelian $\ell$-groups (which capture the structure of the algebras of functions $\C(X,\R)$), MV-algebras, and the infinitary variety $\Delta$ of $\delta$-algebras. Further, we recall the Cignoli-Dubuc-Mundici adjunction between MV-algebras and the dual of  compact Hausdorff spaces, and its restriction to the variety $\Delta$. The logic $\Log$ associated with $\Delta$ is introduced in Section~\ref{s:logic-delta}, where some of its main properties are established. In Section~\ref{s:definability}, we prove that $\Log$ has the Beth definability property and show how the Stone-Weierstrass Theorem can be deduced from it. Finally, in Section~\ref{s:Hilbert-calculus} a logic $\Lo$ is introduced by means of a Hilbert-style calculus, and shown to be strongly complete with respect to a $\int$-valued semantics which, up to a translation between terms and equations, coincides with the consequence relation $\Log$.

\begin{relatedwork}
An algebraic treatment of the Stone-Weierstrass Theorem was provided by Banaschewski~\cite{Banaschewski2001}, who showed that this result is ultimately a consequence of more general properties of the class of $f$-rings. Also, the connection between the Stone-Weierstrass Theorem and the surjectivity of epimorphisms (cf.\ Theorem~\ref{t:equivalent-formulations}), e.g.\ in the category of commutative $\mathrm{C}^*$-algebras, has long been known, cf.~\cite{Reid1970} and also~\cite{BMO2013}. In this article, we focus on the logical gist of the Stone-Weierstrass Theorem and introduce an infinitary equational logic $\Log$ which may be thought of as a ``logic for compact Hausdorff spaces'' (in the same way that classical propositional logic is a logic for zero-dimensional compact Hausdorff spaces by Stone duality for Boolean algebras~\cite{st1}).\footnote{It is well known that the category $\KH$ of compact Hausdorff spaces and continuous maps is not dually equivalent to any Birkhoff variety of algebras. Hence, there is no \emph{finitary} equational logic for compact Hausdorff spaces. For more on the axiomatisability of the dual of $\KH$, see e.g.~\cite{LRVpreprint,MR2017}.} Recently, a modal calculus for compact Hausdorff spaces, based on de Vries duality, was introduced in~\cite{BBSV2019}. For an approach based on multilingual sequent calculus, cf.~\cite{Moshier2004}.
\end{relatedwork}

\begin{notation}
Throughout this paper, we denote continuous functions by $\f,\g,\h$, and reserve the symbols $f,g,h$ for variable assignments or algebra homomorphisms. MV- and $\delta$-algebras, as well as their underlying sets, are denoted by $\A,\B$. If $\leq$ is a lattice order, binary infima and binary suprema are denoted by $\wedge$ and $\vee$, respectively.
\end{notation}

\section{Preliminaries and background}\label{s:prel}
\subsection{Abelian $\ell$-groups and MV-algebras}\label{s:l-groups-MV}
An \emph{Abelian $\ell$-group} is an Abelian group $G$, written additively,
equipped with a lattice order $\leq$ invariant under translations, i.e.,
\[
\forall a,b,c\in G, \ \ a\leq b \ \Longrightarrow \ a+c\leq b+c.
\]
We say that $G$ is \emph{unital} if it is equipped with a distinguished element $u\in G$ (the \emph{unit}) such that, for each $a \in G$, there is an $n\in \N$ such that $a\leq nu$. A prime example of unital Abelian $\ell$-group is $(\R,1)$, where $\R$ is the additive group of real numbers equipped with the usual total order. More generally, for any topological space $X$, the set $\C(X,\R)$ of all continuous $\R$-valued functions on $X$ is an Abelian $\ell$-group with respect to pointwise operations. The constant function $1_X\colon X\to \R$ of value $1$ is a unit for $\C(X,\R)$.

\begin{remark}
Throughout this paper, the expression $\C(X,\R)$ will refer to either the structure of unital Abelian $\ell$-group, or that of topological space with the topology induced by the uniform metric. It will always be clear from the context which structure we are considering.
\end{remark}

Let us denote by $\lGr$ the category of unital Abelian $\ell$-groups and \emph{unital $\ell$-homomorphisms}, i.e.\ functions that are both lattice and group homomorphisms and preserve the unit. Subobjects in $\lGr$, i.e.\ sublattice subgroups containing the unit, are called \emph{unital $\ell$-subgroups}. 

Given an arbitrary unital Abelian $\ell$-group $(G,u)$, we can equip its \emph{unit interval} 
\[
\Gamma(G,u)\coloneqq \{a\in G \mid 0 \leq a \leq u\}
\]
with the operations 
\begin{gather*}
a\oplus b \coloneqq  (a+b)\wedge u \ \text{ and } \ \neg a \coloneqq  u-a.
\end{gather*}
(When the choice of the unit is clear from the context, we write $\Gamma(G)$ instead of $\Gamma(G,u)$.)
The tuple $(\Gamma(G),\oplus,\neg,0)$ carries the structure of an MV-algebra.

An \emph{MV-algebra} is an algebra $(\A,\oplus,\neg,0)$, in the algebraic language $\Lan_{MV}\coloneqq\{\oplus,\neg,0\}$ of type $(2,1,0)$, satisfying the following conditions for all $a,b\in \A$:\footnote{The axiomatisation presented here is redundant: it was shown in~\cite{Kolarik2013} that commutativity of the monoid operation $\oplus$ follows from the other axioms.}
\begin{multicols}{2}
\begin{enumerate}[label=(\roman*), itemsep=0ex]
\item\label{ax-MV1} $(\A,\oplus,0)$ is a commutative monoid
\item\label{ax-MV2} $\neg$ is an involution, i.e.\ $\neg\neg a=a$
\item\label{ax-MV3} $a\oplus \neg 0 =\neg 0$
\item\label{ax-MV4} $\neg(\neg a \oplus b)\oplus b = \neg(\neg b \oplus a)\oplus a$
\end{enumerate}
\end{multicols}
While Boolean algebras are the algebraic counterpart to classical propositional logic, MV-algebras are the algebraic counterpart to {\L}ukasiewicz infinite-valued propositional logic $\Luk$. For more details, we refer the interested reader to~\cite{cdm2000,M11}. 

The operation $\oplus$ should be regarded as a strong disjunction, and the involution $\neg$ plays the role of a negation which allows us to define a strong conjunction by
\[
a\odot b\coloneqq \neg(\neg a\oplus \neg b).
\]
It will be useful to define a further connective:
\begin{align*}
a\ominus b \coloneqq  a\odot \neg b.
\end{align*}
We denote by $\MV$ the category of MV-algebras and \emph{MV-homomorphisms}, i.e.\ functions preserving $\oplus$, $\neg$, and $0$.
Any MV-algebra has an underlying structure of distributive lattice bounded below by $0$ and above by $1\coloneqq\neg 0$. Binary joins (also known as weak disjunctions) are given by
\[
a \vee b = \neg(\neg a \oplus  b)\oplus b.
\]
Thus, item~\ref{ax-MV4} above states that $a\vee b=b\vee a$.
Binary meets (also known as weak conjunctions) are given by the De Morgan condition  $a \wedge b = \neg (\neg a \vee \neg b)$.
Boolean algebras are precisely those MV-algebras that satisfy the law of excluded middle $a\vee\neg a=1$~\cite[Theorems~1.16 and~1.17]{Chang58}.

The \emph{standard MV-algebra} is the real unit interval $\int$ with neutral element $0$ in which the operations $\oplus$ and $\neg$ are defined, respectively, by
\begin{align*}
a\oplus b \coloneqq  \min{\{1,a+b\}}
\end{align*}
(the connective $\oplus$ is frequently referred to as \emph{truncated addition}) and 
\begin{align*}
\neg a\coloneqq 1-a.
\end{align*} 
The derived operations $\odot$ and $\ominus$ are then interpreted as $a\odot b=\max{\{0,a+b-1\}}$ and $a\ominus b= \max{\{0,a-b\}}$ ($\ominus$ is often called \emph{truncated subtraction}).
The underlying lattice order of this MV-algebra coincides with the total order that $\int$ inherits from the real numbers, cf.~\cite[pp.~9--10]{cdm2000}. When we 
refer to $\int$ as an MV-algebra, we always mean the structure just described. 
Note that the standard MV-algebra $\int$ coincides with $\Gamma(\R,1)$, the unit interval of the unital Abelian $\ell$-group $(\R,1)$. Similarly, for any topological space $X$, the unit interval of the Abelian $\ell$-group $\C(X,\R)$ with unit $1_X$ can be identified with the MV-algebra 
\begin{align*}
\C(X,\int)\coloneqq \{\f\colon X\to\int\mid \f \,\text{ is continuous}\},
\end{align*}
where $\int$ is equipped with the Euclidean topology and the MV-algebraic operations of $\C(X,\int)$ are defined pointwise.

Every unital $\ell$-homomorphism $(G,u)\to (G',u')$ between unital Abelian $\ell$-groups restricts to an MV-homomorphism $\Gamma(G,u)\to\Gamma(G',u')$. In fact, this assignment yields a functor $\Gamma\colon \lGr\to\MV$. 
In 1986, Mundici showed that $\Gamma$ is an equivalence of categories.\footnote{We shall only need that $\Gamma$ is full and faithful, hence it reflects isomorphisms. Cf.\ the proof of Theorem~\ref{t:SW-lattice-version} on page~\pageref{proof:SW}. This fact is easier to prove and corresponds to Propositions~3.4 and~3.5 in~\cite{Mundici86}.} For a proof of this fact, and an explicit description of the quasi-inverse functor, see~\cite{Mundici86} or~\cite[Theorems~7.1.2 and~7.1.7]{cdm2000}.
\begin{theorem}[Mundici's Equivalence]\label{t:gamma}
$\Gamma\colon\lGr\to\MV$ is an equivalence of categories.
\end{theorem}
%

\subsection{Ideal theory in MV-algebras}\label{s:ideal-theory}
As in the case of rings, quotients (or congruences) of MV-algebras can be described in terms of ideals. An \emph{ideal} of an MV-algebra $\A$ is a subset of $\A$ that contains $0$, is downwards closed in the lattice order of $\A$, and is closed under $\oplus$.  
The ideal $\langle S\rangle$ generated by a non-empty subset $S\subseteq \A$ can be described as follows (cf.\ \cite[Lemma~1.2.1]{cdm2000}):
\begin{align}\label{eq:generated-ideal}
\langle S\rangle=\{a\in \A\mid a\leq s_1\oplus\cdots\oplus s_n \ \text{for some} \ s_1,\ldots,s_n\in S\}.
\end{align}

Given an ideal $I\subseteq \A$, the corresponding congruence is
\begin{align*}
{\equiv_{I}}\coloneqq \{(a,b)\in \A\times \A \mid \d(a,b) \in I\},
\end{align*}
where $\d$ is the derived operation 
\begin{equation}\label{eq:chang-distance}
\d(a,b)\coloneqq (a\ominus b)\oplus(b\ominus a)
\end{equation}
known as \emph{Chang's distance}. When interpreted in the MV-algebra $\int$, Chang's distance coincides with the usual Euclidean distance. Just observe that, for all $x,y\in\int$,
\[
\d(x,y)=\min\{\max\{0,x-y\}+\max\{0,y-x\},1\}=\begin{cases} y-x \ \ \text{if $x\leq y$} \\ x-y \ \ \text{otherwise} \end{cases} \hspace{-0.8em}=|x-y|.
\]
We write $\A/I$ for the quotient algebra $\A/{\equiv_{I}}$. Conversely, the ideal associated with a congruence $\equiv$ on $\A$ is 
\[
I_{\equiv}\coloneqq \{a\in \A \mid a\equiv 0\}.
\] 
This yields a bijective correspondence between ideals of $\A$ and congruences on $\A$~\cite[Proposition~1.2.6]{cdm2000}. 

A \emph{maximal ideal} of an MV-algebra $A$ is an ideal $\m$ of $A$ that is proper (i.e., $\m\subsetneq \A$) and not strictly contained in any other proper ideal of $A$. The set of maximal ideals of $\A$ is denoted by $\Maxi(\A)$. As a consequence of Zorn's Lemma, an MV-algebra is non-trivial (i.e., it has two distinct elements) precisely when $\Maxi(A)\neq\emptyset$~\cite[Corollary~1.2.15]{cdm2000}.
An MV-algebra is \emph{simple} if it has no proper ideal distinct from $\{0\}$. Up to isomorphism, the simple MV-algebras are precisely the subalgebras of $\int$ \cite[Theorem~3.5.1]{cdm2000}. An MV-algebra $\A$ is \emph{semisimple} if it is a subdirect product of simple MV-algebras; equivalently, if $\bigcap{\Maxi(\A)}=\{0\}$ \cite[Proposition~3.6.1]{cdm2000}. 

Maximal ideals of an MV-algebra $\A$ can also be described in terms of MV-homomorphisms $\A\to\int$. Using the fact that the standard MV-algebra $\int$ is simple, it is not difficult to see that $h^{-1}(0)\in\Maxi(\A)$ for any MV-homomorphism $h\colon \A\to \int$ \cite[Proposition~1.2.16]{cdm2000}. This yields a function
\[
\k\colon \hom_{\MV}(\A,\int)\to\Maxi(A), \ \ \k(h)\coloneqq h^{-1}(0).
\]
For the converse direction, we appeal to H{\"o}lder's Theorem (for a modern proof for $\ell$-groups, see~\cite[\S 2.6]{BigardKeimel77}; for the MV-algebraic version, cf.~\cite[Theorem~3.5.1]{cdm2000}):
\begin{theorem}[H\"older's Theorem]\label{MVHolder}
For any MV-algebra $\A$ and maximal ideal $\m\in\Maxi(\A)$,  there exists a unique homomorphism of MV-algebras $h_{\m}\colon \A\to\int$ such that $\k(h_{\m})=\m$.
\end{theorem}

\subsection{The Cignoli-Dubuc-Mundici adjunction}\label{s:CDM-adjunction}
We recall from~\cite{CDB2004} the Cignoli-Dubuc-Mundici adjunction between the category $\MV$ of MV-algebras and MV-homomorphisms and the dual of the category $\KH$ of compact Hausdorff spaces and continuous maps.

This dual adjunction is induced by the dualising object $\int$, regarded as either a compact Hausdorff space or an MV-algebra. 
Given an MV-algebra $\A$, equip the set  of homomorphisms $\hom_{\MV}(\A,\int)$ with the subspace topology induced by the product topology of $\int^\A$. The ensuing space is denoted by $\Max{\A}$ and referred to as the \emph{maximal spectrum} of $\A$. A routine argument, exploiting the fact that the MV-algebraic operations of $\int$ are continuous with respect to the Euclidean topology, shows that $\Max{\A}$ is closed in the product topology of $\int^\A$. Thus, $\Max{\A}$ is a compact Hausdorff space.

In the MV-algebraic literature (see e.g.~\cite{M11}), the space $\Max{\A}$ is usually studied by means of a different but equivalent representation. Let us endow the set of maximal ideals $\Maxi(A)$ with the \emph{hull-kernel} topology having as a basis of opens the sets of the form
\[
O_a\coloneqq \{\m\in\Maxi(\A)\mid a\notin\m\}
\]
for $a\in \A$. The space $\Maxi(A)$ is compact and Hausdorff, see \cite[Proposition~4.15]{M11}. In fact, the function $\k\colon \Max{\A}\to \Maxi(\A)$, which is a bijection by Theorem~\ref{MVHolder}, is a homeomorphism. For the continuity of $\k$ note that, for any $a\in\A$,
\[
\k^{-1}(O_a)=\{h\in\Max{\A}\mid h(a)\neq 0\}
\]
which is open in $\Max{\A}$ because it coincides with the preimage of $(0,1]$ under the (continuous) projection $\Max{\A}\to \int$, $h\mapsto h(a)$. As every continuous bijection between compact Hausdorff spaces is a homeomorphism, the spaces $\Max{\A}$ and $\Maxi(\A)$ are homeomorphic.

The next lemma provides a characterisation of the closed subsets of $\Max{\A}$ (for a proof, see~\cite[\S 4.4]{M11}). For any subset $S\subseteq \A$, define
\begin{align}\label{eq:vanishing-of-a-set}
\V(S)\coloneqq\{h\in \Max{\A}\mid h(a)=0 \ \text{for all} \ a\in S\}.
\end{align}
 If $a \in \A$, we write $\V(a)$ as a shorthand for $\V(\{a\})$. 
 \begin{lemma}\label{l:closed-sets-spectra}
 The following statements hold for any MV-algebra $\A$ and subsets $S,S'\subseteq \A$.
 \begin{enumerate}[label=\textup{(\alph*)}]
 \item\label{V-inc} $\V(S)\subseteq \V(S')$ if, and only if, $\langle S\rangle \supseteq \langle S'\rangle$.
 \item\label{closed-Max} The closed subsets of $\Max{\A}$ are precisely those of the form $\V(I)$ for some ideal $I\subseteq \A$.
 \end{enumerate}
 \end{lemma}

For every MV-homomorphism $k\colon \A\to \B$, the function 
\[
\Max{k}\colon \Max{\B}\to\Max{\A}, \ \ h\mapsto h\circ k
\] 
is continuous because $(\Max{k})^{-1}(\V(I))=\V(\langle k(I)\rangle)$ for every ideal $I$ of $\A$.
We get a functor
\begin{align*} 
\Max \colon \MV \to \KH^{\rm op}.
\end{align*}
Conversely, given a compact Hausdorff space $X$, let $\C(X,\int)$ be the MV-algebra of all continuous functions $X\to\int$ with the pointwise operations of the standard MV-algebra. If $\f\colon X \to Y$ is a morphism in $\KH$, it is easily seen that the induced map 
\begin{align*}
\C{\f}\colon \C(Y,\int)\to\C(X,\int), \ \ \g\mapsto \g\circ \f
\end{align*}
is a morphism in $\MV$. We can thus regard $\hom_{\KH}(-,\int)$ as a functor
\begin{align*}
\C \colon  \KH^{\rm op}\to  \MV. 
\end{align*}
For every compact Hausdorff space $X$ there is a continuous map
\begin{align*}
\ep_{X} \colon X \to \Max{\C(X,\int)}, \ \ x\mapsto (\C(X,\int)\xrightarrow{\ev_x}\int, \ \f\mapsto \f(x)).
\end{align*}
Moreover, for every MV-algebra $\A$, there is an MV-homomorphism 
\begin{align*}
\eta_{\A} \colon \A \to \C(\Max{\A},\int), \ \ a\mapsto (\Max{\A}\xrightarrow{\ev_a}\int, \ h\mapsto h(a)).
\end{align*}
To improve readability, we will write $\w{a}$ instead of $\ev_a$. Thus, for all $a\in\A$, 
\begin{align}\label{eq:Gelfand-transform}
\w{a}\colon \Max{\A}\to\int
\end{align}
is the continuous map defined by $\w{a}(h)\coloneqq h(a)$.

Denoting by $\id{{\mathscr C}}$ the identity functor on a category $\mathscr{C}$, it is not difficult to see that $\ep_X$ and $\eta_A$ yield natural transformations $\ep \colon \Max\circ\C \to \id{\KH^{\rm op}}$ and $\eta \colon \id{\MV} \to \C\circ\Max$, respectively. 
\begin{theorem}[Cignoli-Dubuc-Mundici Adjunction]\label{t:Max-C-adj}
The natural transformations $\eta$ and $\ep$ are the unit and counit, respectively, of an adjunction $\Max\dashv\C\colon \KH^{\rm op}\to \MV$. Furthermore, the following statements hold:
\begin{enumerate}[label=\textup{(\alph*)}]
\item\label{counit} For every $X$ in $\KH$, the component $\ep_X\colon X\to \Max{\C(X,\int)}$ of the counit is a homeomorphism, i.e.\ the functor $\C$ is full and faithful.
\item\label{unit-sem} For every $\A$ in $\MV$, the component $\eta_{\A}\colon \A\to\C(\Max{\A},\int)$ of the unit is injective if, and only if, $\A$ is semisimple.
\end{enumerate}
\end{theorem}
The adjunction $\Max\dashv\C$ was introduced in~\cite{CDB2004}.  
For a proof of items~\ref{counit} and~\ref{unit-sem}, see e.g.\ \cite[Theorem~4.16]{M11}.

\subsection{The variety $\Delta$}\label{subs:delta}
Any adjunction restricts to an equivalence between the full subcategories defined by the fixed objects, i.e.\ those objects for which the components of the unit and counit, respectively, are isomorphisms. In the case of the Cignoli-Dubuc-Mundici adjunction \[\Max\dashv\C\colon \KH^\op\to\MV,\] as the counit is a natural isomorphism by Theorem~\ref{t:Max-C-adj}\ref{counit}, we see that there exists a full subcategory of $\MV$ dually equivalent to $\KH$. In~\cite{MR2017}, the variety $\Delta$ of $\delta$-algebras was defined and shown to be isomorphic to a full subcategory of $\MV$ dually equivalent to $\KH$. However, the proof of the duality between $\Delta$ and $\KH$ given there relies on the Stone-Weierstrass Theorem. In this section we recall some facts about $\delta$-algebras that do not depend on the Stone-Weierstrass Theorem. 

Consider the algebraic language $\Lan_{\Delta}\coloneqq\{\delta,\oplus,\neg,0\}$ of type $(\omega,2,1,0)$.
As the operation $\delta$ takes as argument a countably infinite sequence of terms, we write $\vec{x},\vec{y}$ and $\vec{0}$ as shorthands for the sequences $x_1,x_2,\ldots$, $y_1,y_2,\ldots$, and $0,0,\ldots$, respectively. It will be convenient to introduce a derived unary operation $\fm$, to be thought of as multiplication by $\frac{1}{2}$:
\begin{align*}
\fm(x) \coloneqq  \delta(x,\vec{0}).
\end{align*}
Moreover, recall from equation~\eqref{eq:chang-distance} that $\d$ denotes Chang's distance. 
\begin{definition}\label{d:delta}
 A \emph{$\delta$-algebra} is an $\Lan_{\Delta}$-algebra $(\A,\delta,\oplus,\neg,0)$ such that $(\A,\oplus,\neg,0)$ is an MV-algebra and the following identities are satisfied for every $x,y\in \A$ and $\vec{x},\vec{y}\in \A^\omega$: 
\begin{multicols}{2}
\begin{enumerate}[label=(\roman*), itemsep=0em]
\item\label{Delta1} $\d\left(\delta(\vec{x}),\delta(x_1,\vec{0})\right)=\delta(0,x_2,x_3,\ldots)$
\item\label{Delta2} $\fm(\delta(\vec{x}))=\delta(\fm(x_1),\fm(x_2),\ldots)$
\item\label{Delta3} $\delta\left(x,x,\ldots\right)=x$
\item\label{Delta4} $\delta(0,\vec{x})=\fm(\delta(\vec{x}))$
\item\label{Delta5} $\delta(x_1,x_2,\ldots)\leq \delta(x_1\oplus y_1,x_2\oplus y_2,\ldots)$
\item\label{Delta6} $\fm(x\ominus y)=\fm(x)\ominus\fm(y)$
\end{enumerate}
\end{multicols}
\end{definition}
A homomorphism of $\delta$-algebras, or \emph{$\delta$-homomorphism} for short, is a homomorphism of the underlying MV-algebras that preserves the operation $\delta$. In fact, it will follow from Theorem~\ref{t:semisimple-and-full}\ref{Delta-full} below that all MV-homomorphisms between $\delta$-algebras are $\delta$-homomorphisms. We write $\Delta$ for the category (as well as for the variety) of $\delta$-algebras and $\delta$-homomorphisms.
\begin{remark}
Since the operation $\delta$ has infinite arity, $\Delta$ is not a variety of Birkhoff algebras. Thus, we rely on the theory of varieties of infinitary algebras as developed by S{\l}omi{\'n}ski in~\cite{Slominski59}, see also~\cite{Linton66}. In the following sections we will not need to exploit the axiomatisation in Definition~\ref{d:delta}. Instead, we will make use of the properties of $\Delta$ summarised in Theorem~\ref{t:semisimple-and-full}.
\end{remark}
The operation $\delta$, and its semantic interpretation which we now recall, were introduced by Isbell in~\cite{Isbell}.
In the unit interval $\int$, for every sequence $\vec{x}\in\int^\omega$, set
\[\delta(\vec{x})\coloneqq\sum_{i=1}^{\infty}{\frac{x_i}{2^i}}.\]
It is not difficult to see that the standard MV-algebra $\int$, equipped with this interpretation of $\delta$, is a $\delta$-algebra in which the unary operation $\fm$ coincides with multiplication by $\frac{1}{2}$.
More generally, for every compact Hausdorff space $X$, the MV-algebra $\C(X,\int)$ is a $\delta$-algebra if, for all $\vec{\g}\in \C(X,\int)^{\omega}$,  $\delta(\vec{\g})$ is defined as the uniformly convergent series 
\begin{equation}\label{eq:interp-delta-C(X)}
\delta(\vec{\g})\coloneqq \sum_{i=1}^{\infty}\frac{\g_i}{2^i}.
\end{equation}
Throughout this paper, whenever we regard $\C(X,\int)$ as a $\delta$-algebra, we assume that the interpretation of the operation $\delta$ is the one given above. In fact, there is no other structure of $\delta$-algebra expanding the pointwise MV-algebraic structure of $\C(X,\int)$ \cite[Corollary~6.4]{MR2017}. 

Note that the functor $\C\colon\KH^{\op}\to\MV$ factors through the forgetful functor $\Delta\to\MV$. Hence, the Cignoli-Dubuc-Mundici adjunction restricts to an adjunction 
\[
\Max\dashv\C\colon\KH^{\op}\to\Delta.
\]
\begin{theorem}\label{t:semisimple-and-full}
The following statements hold:
\begin{enumerate}[label=\textup{(\alph*)}]
\item\label{Delta-sem} The underlying MV-algebra of any $\delta$-algebra is semisimple.
\item\label{Delta-full} The forgetful functor $\Delta\to\MV$ is full.
\item\label{Delta-unit} For every $\A\in\Delta$, the map $\eta_{\A}\colon \A\to\C(\Max{\A},\int)$ is an injective $\delta$-homomorphism.
\end{enumerate}
\end{theorem}
\begin{proof}
The first two items are Theorems~5.5 and~6.3, respectively, in~\cite{MR2017}. The third item follows at once from the first two, along with Theorem~\ref{t:Max-C-adj}\ref{unit-sem}.
\end{proof}
\begin{remark}\label{rem:epis-to-monos}
Let $\Set$ be the category of sets and functions. In view of Theorem~\ref{t:semisimple-and-full}\ref{Delta-full}, the functor $\hom_{\Delta}(-,\int)\colon \Delta^\op\to\Set$ coincides with the composition of $\Max\colon \Delta^\op\to \KH$ with the forgetful functor $\KH\to\Set$. As the functor $\hom_{\Delta}(-,\int)$ is representable, hence it preserves limits, and the forgetful functor $\KH\to\Set$ reflects limits, we see that $\Max\colon \Delta^\op\to \KH$ preserves all limits. For instance, $\Max$ sends epis in $\Delta$ to monos in $\KH$. In particular, if $h\colon \A\epi \B$ is a surjective homomorphism in $\Delta$, then the continuous map $\Max{h}$ identifies $\Max{\B}$ with a closed subspace of $\Max{\A}$. 
\end{remark}

\section{The logic $\Log$}\label{s:logic-delta}
The purpose of this section is to introduce the infinitary equational logic $\Log$ and prove its main properties: the compactness and local deduction theorems, and some of their consequences. Every Birkhoff variety $\mathcal{V}$ of finitary algebras comes with an associated logic, namely the equational consequence $\vDash_{\mathcal{V}}$. See, e.g., \cite[\S 2]{MMT2014}. This concept makes sense also for varieties of infinitary algebras in the sense of S{\l}omi{\'n}ski~\cite{Slominski59}. We shall spell this out in the case of the variety $\Delta$. 

To improve readability, write $\Lan$ instead of $\Lan_{\Delta}$ for the language of $\delta$-algebras. Given a (possibly infinite) set of propositional variables $\x$, let $\T(\x)$ denote the \emph{algebra of $\Lan$-terms in the variables~$\x$}. In other words, $\T(\x)$ is the \emph{absolutely free $\Lan$-algebra over $\x$}. As in the case of varieties of Birkhoff algebras, the free $\delta$-algebra over $\x$, denoted by $\F(\x)$, can be constructed in a canonical way as a quotient 
\[
\uprho\colon \T(\x)\epi \F(\x).
\]
(This is true for any variety of infinitary algebras, provided there is a cardinal $\kappa$ such that all operations have arity smaller than $\kappa$, cf.~\cite[Chapter~III]{Slominski59}.) If $\A$ is an $\Lan$-algebra and $f\colon \x\to \A$ is a function, also called an \emph{assignment} of the variables $\x$ in $\A$, we denote by  
\[
\li{f}\colon \T(\x)\to \A
\] 
the unique $\Lan$-homomorphism extending $f$. 

An \emph{$\Lan$-equation} in the variables $\x$ is a pair $(s,t)\in \T(\x)\times \T(\x)$ of $\Lan$-terms; we shall use the more suggestive notation $s\approx t$ for the equation $(s,t)$. 
Arbitrary $\Lan$-equations will be denoted by $\sigma$, and sets of $\Lan$-equations by $\Si$, $\Ga$, or $\Pi$. To emphasize that the variables of an $\Lan$-term, $\Lan$-equation, or set of $\Lan$-equations, are contained in $\x$, we write $s(\x), \sigma(\x)$, or $\Si(\x)$. Further, for convenience of notation, we drop reference to the language $\Lan$ when speaking of $\Lan$-terms, $\Lan$-equations, and $\Lan$-homomorphisms.
Given a set of equations $\Si(\x)$, an $\Lan$-algebra $\A$ and an assignment $f\colon \x\to \A$, we define
\[
\A,f\models \Si \ \text{ if, and only if, } \ \Si\subseteq \ker{\li{f}}
\] 
where $\ker{\li{f}}\coloneqq \{(s,t)\in \T(\x)\times \T(\x)\mid \li{f}(s)=\li{f}(t)\}$ is the kernel of $\li{f}$.
If $\A,f\models \Si$, we say that $\Si$ is \emph{satisfied} in $\A$ with respect to the variable assignment $f$.
For any set of equations $\Si(\x)\cup\{\sigma(\x)\}$, define $\Si\Log\sigma$ if and only if for every $\A\in\Delta$ and assignment $f\colon \x\to \A$,
\[
A,f\models\Si \ \ \Longrightarrow \ \ \A,f\models \sigma.
\]
Finally, given a set of equations $\Si(\x)\cup\Ga(\x)$, the equational consequence relation $\Log$ is defined by $\Si\Log\Ga$ if, and only if, $\Si\Log \sigma$ for every $\sigma\in\Ga$. 

As with groups, every equation $s\approx t$ in the language of $\delta$-algebras is equivalent to one of the form $s'\approx 0$. Just observe that, for any $\A\in\Delta$ and all elements $a,b\in \A$,
\[
a=b \ \Longleftrightarrow \ \d(a,b)=0
\] 
where $\d$ is Chang's distance. By applying the involution $\neg$, we can also transform the equation $s\approx  t$ into an equivalent one of the form $s''\approx 1$.
Furthermore, since every member of $\Delta$ embeds into a power of $\int$ by Theorem~\ref{t:semisimple-and-full}\ref{Delta-unit}, we see that
\begin{equation}\label{eq:int-generates}
\Si\Log\sigma \ \Longleftrightarrow \ \text{for every assignment $f\colon \x\to \int$}, \ \int,f\models\Si \ \ \text{entails} \ \ \int,f\models \sigma.
\end{equation}
In other words, the logic $\Log$ is complete with respect to valuations in $\int$ (algebraically: $\int$ generates the variety $\Delta$).

An important observation is that the logic $\Log$ is compact because the corresponding maximal spectra are topologically compact (cf.~Lemma~\ref{l:compactness-logic-delta} below). We refer to an arbitrary set of equations $\Si(\x)$ as a \emph{theory}, and say that $\Si(\x)$ is \emph{satisfiable} if there exist a non-trivial algebra $\A\in\Delta$ (i.e., a $\delta$-algebra with at least two distinct elements) and an assignment $f\colon \x\to \A$ such that $\A,f\models \Si$. Recalling from Section~\ref{s:ideal-theory} that $\A$ is non-trivial if, and only if, it admits a maximal ideal, we see that $\Si$ is satisfiable precisely when there exists an assignment $f\colon \x\to\int$ satisfying $\int,f\models \Si$. Here and throughout this paper, by an \emph{ideal} of a $\delta$-algebra we understand an ideal of the underlying MV-algebra.
\begin{lemma}[Compactness]\label{l:compactness-logic-delta}
Let $\Si(\x)$ be any theory. Then $\Si(\x)$ is satisfiable if, and only if, all its finite subsets are satisfiable.
\end{lemma}
\begin{proof}
For the non-trivial direction, suppose that every finite subset of $\Si(\x)$ is satisfiable. We can assume without loss of generality that 
\[
\Si(\x)=\{s_i\approx 0\mid i\in I\}
\] 
for a set of terms $\{s_i(\x)\mid i\in I\}$. With the notation in~\eqref{eq:vanishing-of-a-set}, each term $s_i$ yields a closed subset $\V(\quot(s_i))$ of $\Max{\F(\x)}$, where $\quot\colon \T(\x)\epi\F(\x)$ is the canonical quotient. We claim that
\[
\{\V(\quot(s_i))\mid i\in I\}
\] 
has the finite intersection property. For any finite subset $I_0\subseteq I$, the theory 
$
\Si_0\coloneqq\{s_i\approx 0\mid i\in I_0\} 
$
is satisfiable, hence there exists an assignment $f\colon\x\to\int$ satisfying $\int,f\models \Si_0$. Denoting by $h\colon \F(\x)\to \int$ the unique homomorphism extending $f$, we have
\[
h(\quot(s_i))=\li{f}(s_i)=0 \ \text{ for every $i\in I_0$},
\]
i.e., $h\in \bigcap{\{\V(\quot(s_i))\mid i\in I_0\}}$. It follows that $\{\V(\quot(s_i))\mid i\in I\}$ has the finite intersection property and so, as the space $\Max{\F(\x)}$ is compact, $\bigcap{\{\V(\quot(s_i))\mid i\in I\}}\neq\emptyset$. Therefore, there exists a homomorphism $k\colon \F(\x)\to \int$ such that $k(\quot(s_i))=0$ for every $i\in I$. We get $\int,k\circ \quot_{\restriction\x}\models \Si$, showing that the theory $\Si(\x)$ is satisfiable.
\end{proof}

\begin{remark} 
The compactness of $\Log$ can be equivalently stated as the fact that $\Si\Log 0\approx 1$ entails $\Si_0\Log 0\approx 1$ for some finite subset $\Si_0\subseteq \Si$.
However, it is not the case that, for any equation $\sigma$, if $\Si\Log\sigma$ then $\Si_0\Log\sigma$ for some finite subset $\Si_0\subseteq \Si$. For instance, let $\Si\coloneqq\{x_i\approx 0\mid i\in \omega\}$ and $\sigma\coloneqq\{\delta(\vec{x})\approx 0\}$, where $\vec{x}=x_1,x_2,x_3,\ldots$. By axiom~\ref{Delta3} in Definition~\ref{d:delta}, we have $\delta(\vec{0})=0$, hence $\Si\Log \sigma$. However, it is not difficult to see that $\Si_0\not\Log \sigma$ for each finite subset $\Si_0\subseteq \Si$. To see this, let $\Si_0$ be a finite subset of $\Si$ and pick $j\in \omega$ such that the equation $x_j\approx 0$ does not belong to $\Si_0$. If $\x\coloneqq\{x_i\mid i\in\omega\}$, the assignment $f\colon \x\to \int$ that is $1$ on $x_j$ and $0$ elsewhere satisfies $\int, f\models \Si_0$ and $\int, f\not\models \sigma$.
\end{remark}

Next, we prove a local deduction theorem for $\Log$, analogous to the one for {\L}ukasiewicz infinite-valued logic. To this end recall that, in any MV-algebra, the operation $\odot$ admits an upper adjoint $\to$, i.e.\ $a\odot b\leq c \Leftrightarrow a\leq b\to c$ \cite[p.~86]{cdm2000}. Explicitly, $b\to c=\neg b\oplus c$. For every $k\in \N$ and term $s$, write 
\[
ks\coloneqq\underbrace{s\oplus \cdots\oplus s}_{\text{$k$ times}} \ \text{ and } \ s^k\coloneqq\underbrace{s\odot \cdots\odot s}_{\text{$k$ times}}.
\]
An elementary computation shows that $k(\neg s)=\neg(s^k)$.
\begin{lemma}[Local Deduction Theorem]\label{l:deduction}
Let $\Si(\x)$ be a theory, and $s(\x),t(\x)$ two terms such that $\Si\cup\{s\approx 1\}\Log t\approx 1$. Then there is a $k\in\N$ such that $\Si\Log (s^k\to t)\approx 1$.
\end{lemma}
\begin{proof}
The same proof as for {\L}ukasiewicz logic (cf.~\cite[Propositions~4.2.9 and~4.6.4]{cdm2000}), hinging on the ideal theory in MV-algebras, applies here mutatis mutandis. 
We spell out the details for the sake of completeness. Assume without loss of generality that 
\[
\Si(\x)=\{u_i\approx 0\mid i\in I\}
\] 
for some set of terms $\{u_i(\x)\mid i\in I\}$ and let $h\colon \F(\x)\to\int$ be any homomorphism such that $h(\quot(s))=0$ and $h(\quot(u_i))=1$ for all $i\in I$.  
Then, 
\[\int,h\circ \quot_{\restriction \x}\models \Si\cup\{s\approx 1\}
\] 
and, because $\Si\cup\{s\approx 1\}\Log t\approx 1$, we see that $\int,h\circ \quot_{\restriction \x}\models t\approx 1$. In other words, $h(\quot(t))=1$. Applying the involution $\neg$, we obtain
\[
\V(\{\quot(u_i)\mid i\in I\}\cup\{\neg \quot(s)\})\subseteq \V(\neg \quot(t)).
\]
In view of Lemma~\ref{l:closed-sets-spectra}\ref{V-inc}, $\neg \quot(t)$ belongs to the ideal of $\F(\x)$ generated by ${\{\quot(u_i)\mid i\in I\}}\cup\{\neg \quot(s)\}$. By equation~\eqref{eq:generated-ideal}, there exist $\phi_1,\ldots,\phi_n\in \{\quot(u_i)\mid i\in I\}$ and a $k\in \N$ such that
\begin{align*}
\neg \quot(t)&\leq \phi_1\oplus\cdots\oplus \phi_n\oplus k(\neg \quot(s)) \\
&= \phi_1\oplus\cdots\oplus \phi_n\oplus \neg(\quot(s)^k) \\
&= \quot(s)^k\to(\phi_1\oplus\cdots\oplus \phi_n),
\end{align*}
which yields $\quot(s)^k\odot \neg \quot(t)\leq \phi_1\oplus\cdots\oplus\phi_n$. So, $\quot(s)^k\odot \neg \quot(t)$ belongs to the ideal of $\F(\x)$ generated by $\{\quot(u_i)\mid i\in I\}$.
We claim that 
\[
\Si\Log (s^k\to t)\approx 1.
\] 
Let $\A\in\Delta$, and $f\colon \x\to \A$ an assignment satisfying $\A,f\models \Si$. If $g\colon \F(\x)\to\A$ is the unique homomorphism extending $f$, then $g^{-1}(0)$ is an ideal of $\F(\x)$ containing $\quot(u_i)$ for every $i\in I$. Since $\quot(s)^k\odot \neg \quot(t)$ belongs to the ideal generated by $\{\quot(u_i)\mid i\in I\}$, we get $\quot(s)^k\odot \neg \quot(t)\in g^{-1}(0)$. Therefore,
\[
\quot(s)^k\to \quot(t)=\neg \quot(s)^k\oplus \quot(t)=\neg(\quot(s)^k\odot \neg \quot(t))\in g^{-1}(1).
\]
We conclude that
\[
\li{f}(s^k\to t)=(g\circ \quot)(s^k\to t)=g(\quot(s)^k\to \quot(t))=1,
\]
i.e.\ $\A,f\models (s^k\to t)\approx 1$. This settles the lemma.
\end{proof}
Next, let us observe that for every term $s\in \T(\x)$ and real number $r\in \int$ there exists a term $\u{r}s\in \T(\x)$ such that, for all $h\in \Max{\F(\x)}$,
\begin{equation}\label{eq:real-scalars}
\w{\quot(\u{r}s)}(h)=r\cdot \big(\w{\quot(s)}(h)\big),
\end{equation}
where $\w{\quot(\u{r}s)},\w{\quot(s)}\colon \Max{\F(\x)}\to\int$ are the continuous functions defined in~\eqref{eq:Gelfand-transform}.
In other words, the multiplication by real scalars in $\int$ is definable in the language of $\delta$-algebras. If $s$ is the constant $1$, we write $\u{r}$ instead of $\u{r}1$. To define the term $\u{r}s$, consider a binary expansion $\vec{r}\in \{0,1\}^\omega$ of $r$ and let 
\[
\vec{t}\coloneqq (t_i)_{i\in \omega}\in \T(\x)^\omega \ \text{ where } \ t_i\coloneqq\begin{cases} s &\mbox{if } r_i=1 \\
0 & \mbox{otherwise.} \end{cases}
\]
It follows from equation~\eqref{eq:interp-delta-C(X)} that the term $\u{r}s\coloneqq \delta(\vec{t})$ satisfies the desired property.

This yields, in particular, an elementary proof of the following fact:
\begin{proposition}\label{p:spectrum-of-free}
For any set $\x$ there exists a homeomorphism $\Max{\F(\x)}\cong \int^{\x}$, where the Tychonoff cube $\int^{\x}$ is equipped with the product topology.
\end{proposition} 
\begin{proof}
Since the functor $\Max\colon \Delta^\op\to \KH$ sends coproducts in $\Delta$ to products in $\KH$ by Remark~\ref{rem:epis-to-monos}, and $\F(\x)$ is the coproduct of $\x$ copies of the $\delta$-algebra $\F(x)$ free on one generator, it suffices to show that $\Max{\F(x)}\cong \int$. Let $\nu\colon \Max{\F(x)}\to\int$ be the function sending a homomorphism $h\colon \F(x)\to\int$ to $h(\quot(x))$. Then $\nu$ is clearly a bijection. Further, for any $\epsilon,\epsilon'\in\int$, 
\begin{align*}
\nu^{-1}([\epsilon,\epsilon'])&=\{h\in\Max{\F(x)}\mid \epsilon\leq h(\quot(x))\leq \epsilon'\} \\
&=\{h\in\Max{\F(x)}\mid \epsilon\ominus h(\quot(x))=0=h(\quot(x))\ominus \epsilon'\} \\
&=\{h\in\Max{\F(x)}\mid h(\quot(\u{\epsilon})\ominus \quot(x))=0=h(\quot(x)\ominus \quot(\u{\epsilon'}))\} \\
&=\V(\quot(\u{\epsilon})\ominus \quot(x))\cap \V(\quot(x)\ominus \quot(\u{\epsilon'}))
\end{align*}
which is a closed subset of $\Max{\F(x)}$. Hence, $\nu$ is a continuous bijection. As every continuous bijection between compact Hausdorff spaces is a homeomorphism, the statement follows.
\end{proof}
We now state and prove a Robinson's Joint Consistency Theorem for the logic $\Log$, which will allow us to prove a useful interpolation result in Corollary~\ref{c:interpolation-delta} below. (In fact, the desired interpolation result follows from the special case of the Joint Consistency Theorem where the two theories share the same set of propositional variables.)
\begin{theorem}\label{th:logical-urysohn}
For any two theories $\Si_1(\x,\y),\Si_2(\y,\z)$, the union $\Si_1\cup \Si_2$ is satisfiable if, and only if, there is no term $s(\y)$ such that $\Si_1\Log s\approx 0$ and $\Si_2\Log s\approx 1$.
\end{theorem}
\begin{proof} 
If either $\Si_1$ or $\Si_2$ are unsatisfiable, there is nothing to prove. Hence, assume $\Si_1,\Si_2$ are satisfiable. 
Clearly, if the union $\Si_1\cup \Si_2$ is satisfiable, there is no term $s(\y)$ satisfying $\Si_1\Log s\approx 0$ and $\Si_2\Log s\approx 1$. Hence, suppose that $\Si_1\cup \Si_2$ is unsatisfiable.

Let us assume first that $\x=\z=\emptyset$, and so $\Si_1$ and $\Si_2$ are theories in the variables~$\y$.
By Lemma~\ref{l:compactness-logic-delta}, there exists a finite subset $\{\sigma_1,\ldots,\sigma_n\}\subseteq \Si_2$ such that $\Si_1\cup\{\sigma_1,\ldots,\sigma_n\}$ is unsatisfiable. Suppose that each equation $\sigma_i$, for $i\in\{1,\ldots,n\}$, is of the form $t_i\approx 1$ for some term $t_i(\y)$, and define the term 
\[
u(\y)\coloneqq t_1\wedge\cdots\wedge t_n.
\] 
Then $\Si_1\cup \{u\approx 1\}\Log 0\approx 1$.
In view of Lemma~\ref{l:deduction}, there exists a $k\in\N$ such that $\Si_1\Log (u^k\to 0)\approx 1$, that is $\Si_1\Log \neg(u^k)\approx 1$. Therefore, $s\coloneqq u^k$ satisfies $\Si_1\Log s\approx 0$. Further, $\Si_2\Log u\approx 1$ entails $\Si_2\Log s\approx 1$.

For the general case, let us assume without loss of generality that 
\[
\Si_1=\{s_i(\x,\y)\approx 0\mid i\in I\} \ \text{ and } \ \Si_1=\{t_j(\y,\z)\approx 0\mid j\in J\}.
\] 
Let $S_1\coloneqq \{\quot(s_i)\mid i\in I\}\subseteq \F(\x,\y)$ and $S_2\coloneqq \{\quot(t_j)\mid j\in J\}\subseteq \F(\y,\z)$. By Proposition~\ref{p:spectrum-of-free}, $\V(S_1)$ and $\V(S_2)$ can be identified with closed subsets of the Tychonoff cubes $\int^{\x,\y}$ and $\int^{\y,\z}$, respectively. When $\Si_1$ and $\Si_2$ are regarded as theories in the variables $\x,\y,\z$, the corresponding closed subsets of $\int^{\x,\y,\z}$ are the cylinders $\V(S_1)\times\int^{\z}$ and $\int^{\x}\times \V(S_2)$. As $\Si_1\cup \Si_2$ is unsatisfiable, these two cylinders are disjoint. But then
\[
\pi_{\y}(\V(S_1)\times\int^{\z})\cap \pi_{\y}(\int^{\x}\times \V(S_2))=\emptyset,
\]
where $\pi_{\y}\colon \int^{\x,\y,\z}\twoheadrightarrow \int^{\y}$ is the projection. Just observe that, if $b\in\int^{\y}$ is both in $\pi_{\y}(\V(S_1)\times\int^{\z})$ and $\pi_{\y}(\int^{\x}\times \V(S_2))$, then there exist $a\in \int^{\x}$ and $c\in \int^{\z}$ such that $\{(a,b)\}\times\int^{\z}\subseteq \V(S_1)\times\int^{\z}$ and $\int^{\x}\times\{(b,c)\}\subseteq \int^{\x}\times \V(S_2)$. So, $(a,b,c)$ is in the intersection of the cylinders $\V(S_1)\times\int^{\z}$ and $\int^{\x}\times \V(S_2)$, a contradiction. 

Now, since $\pi_{\y}(\V(S_1)\times\int^{\z})$ is a closed subset of $\int^{\y}$, there exists a set $T_1\subseteq \F(\y)$ such that $\pi_{\y}(\V(S_1)\times\int^{\z})=\V(T_1)$. Similarly, $\pi_{\y}(\int^{\x}\times \V(S_2))=\V(T_2)$ for some set $T_2\subseteq \F(\y)$. Define the theories
\[
\tilde{\Si}_1(\y)\coloneqq \{s'(\y)\approx 0\mid s'\in \quot^{-1}(T_1)\} \ \text{ and } \ \tilde{\Si}_2(\y)\coloneqq \{t'(\y)\approx 0\mid t'\in \quot^{-1}(T_2)\}.
\]
Because $\V(T_1)\cap \V(T_2)=\emptyset$, the theory $\tilde{\Si}_1 \cup \tilde{\Si}_2$ is unsatisfiable. By the first part of the proof, there exists a term $s(\y)$ such that $\tilde{\Si}_1 \Log s\approx 0$ and $\tilde{\Si}_2 \Log s\approx 1$. Using the fact that $\pi_{\y}(\V(S_1)\times\int^{\z})=\V(T_1)$, it is not difficult to see that $\Si_1\Log \sigma$ for every $\sigma\in\tilde{\Si}_1$; similarly, $\Si_2\Log \sigma$ for every $\sigma\in \tilde{\Si}_2$. Thus, we see that $\Si_1 \Log s\approx 0$ and $\Si_2 \Log s\approx 1$.
\end{proof}

Note that with any subset $J\subseteq\F(\x)$ we can associate a theory
\begin{equation}\label{eq:Si[J]}
\Th{J} \coloneqq \{s(\x)\approx 0\mid s\in \quot^{-1}(J)\}
\end{equation}
in the variables $\x$.
The theories of the form $\Th{J}$, for $J$ an ideal of $\F(\x)$, will play a crucial role in the following.
\begin{remark}\label{rem:Urysohn}
By Lemma~\ref{l:closed-sets-spectra}\ref{closed-Max} and Proposition~\ref{p:spectrum-of-free}, every closed subspace of a Tychonoff cube $\int^{\x}$ is homeomorphic to $\V(J)$ for some ideal $J$ of $\F(\x)$. Translating from ideals to theories, cf.\ equation~\eqref{eq:Si[J]}, Theorem~\ref{th:logical-urysohn} then yields the following Urysohn's Lemma for Tychonoff cubes: Given disjoint closed sets $C_1,C_2\subseteq \int^{\x}$, there exists a continuous function $\f\colon \int^{\x}\to\int$ satisfying $\f_{\restriction C_1}=0$ and $\f_{\restriction C_2}=1$. Just observe that, for any two ideals $J_1,J_2$ of $\F(\x)$, $\V(J_1)\cap \V(J_2)=\emptyset$ precisely when the theory $\Th{J_1}\cup \Th{J_2}$ is unsatisfiable.
\end{remark}

\begin{corollary}\label{c:interpolation-delta}
Let $\A\in\Delta$, $C_1,\ldots,C_n$ pairwise disjoint closed subsets of $\Max{\A}$, and $r_1,\ldots,r_n\in\int$. Then there exists $a\in \A$ satisfying $\w{a}_{\restriction C_i}=r_i$ for every $i\in\{1,\ldots,n\}$.
\end{corollary}
\begin{proof}
Let $\x$ be a set such that there exists a surjective homomorphism $f\colon \F(\x)\epi \A$. By Remark~\ref{rem:epis-to-monos}, the maximal spectrum $\Max{\A}$ can be identified with a closed subspace of $\Max{\F(\x)}$. Hence, by Lemma~\ref{l:closed-sets-spectra}\ref{closed-Max}, each closed set $C_i$, for $i\in\{1,\ldots,n\}$, is of the form $\V(J_i)\cap \Max{\A}$ for some ideal $J_i$ of $\F(\x)$. With the notation of~\eqref{eq:Si[J]}, for each $i\in\{1,\ldots,n\}$ we consider the theory $\Th{J_i}$. We have 
\[
\V(\langle J_1\cup\cdots\cup J_n\cup f^{-1}(0)\rangle)=\V(J_1)\cap\cdots\cap \V(J_n)\cap \Max{\A}=C_1\cap\cdots\cap C_n=\emptyset,
\]
i.e.\ $\langle J_1\cup\cdots\cup J_n\cup f^{-1}(0)\rangle$ is the improper ideal of $\F(\x)$.
Setting $\Gamma(\x)\coloneqq\ker(f\circ\quot)$, we see that the theory 
\[
\Th{J_1}\cup \cdots\cup\Th{J_n} \cup \Gamma
\] 
is unsatisfiable. By Theorem~\ref{th:logical-urysohn}, for each $i\in\{1,\ldots,n\}$ there exists a term $s_i$ such that 
\[
\Th{J_i}\cup \Gamma\Log s_i\approx 1 \ \text{ and } \ \bigcup_{j\neq i}{\Th{J_j}}\cup \Gamma\Log s_i\approx 0.
\] 
That is, $\w{\quot(s_i)}_{\restriction C_i}=1$ and $\w{\quot(s_i)}_{\restriction C_j}=0$ whenever $j\neq i$. 
If $a\in \A$ is the image of the term $\u{r_1}s_1\vee\cdots\vee \u{r_n}s_n$ under $f\circ \quot\colon \T(\x)\epi \A$ then, using the fact that $C_1,\ldots,C_n$ are pairwise disjoint, we see that $\w{a}_{\restriction C_i}=r_i$ for every $i\in\{1,\ldots,n\}$.
\end{proof}

\section{The Beth definability property}\label{s:definability}
In this section we prove that the logic $\Log$ has the Beth definability property, asserting that implicit definability is equivalent to explicit definability. We then derive the Stone-Weierstrass Theorem from the Beth definability property of $\Log$.
\begin{definition}
Consider a set of variables $\x$, a variable $y$ not in $\x$, and a theory $\Si(\x,y)$. For any variable $z$, write $\Si(\x,z)$ for the theory obtained from $\Si(\x,y)$ by replacing $y$ by $z$. We say that $\Si$ \emph{implicitly defines} $y$ over $\x$ if, for every variable $z$,
\[
\Si(\x,y)\cup\Si(\x,z)\Log y\approx z. 
\] 
Further, $\Si$ \emph{explicitly defines} $y$ over $\x$ if there exists a term $s_y(\x)$  such that
\[
\Si(\x,y)\Log y\approx s_y.
\]
\end{definition}
The meaning of implicit definability is that any assignment $f\colon \x\to \A$ into an algebra $\A\in\Delta$ can be extended to at most one assignment $g\colon \x,y\to\A$ satisfying $\A,g\models \Si$. On the other hand, an explicit definition $s_y$ of $y$ witnesses the fact that the interpretation of $y$ in a model of $\Si(\x,y)$ is completely determined by the interpretation of $\x$. Clearly, if $\Si$ explicitly defines $y$ over $\x$, then it implicitly defines $y$ over $\x$. The Beth definability property states that the converse holds as well.
\begin{definition}
The logic $\Log$ has the \emph{Beth definability property} if, whenever a theory $\Si(\x,y)$ implicitly defines $y$ over $\x$, then $\Si(\x,y)$ explicitly defines $y$ over $\x$.
\end{definition}

\begin{remark}
The definition of Beth definability property given above is the direct specialisation of the classical notion for first-order logic, see e.g.\ \cite[p.~90]{CK1990}, obtained by replacing relation symbols by propositional variables. This definition, employed e.g.\ in \cite[\S 5.6]{HMT1985} or~\cite{Makkai1993}, corresponds to the \emph{singleton} Beth property in abstract algebraic logic~\cite{Hoogland2000}.
\end{remark}

Following Proposition~\ref{p:spectrum-of-free}, throughout this section we identify a Tychonoff cube $\int^{\x}$ with the maximal spectrum $\Max{\F(\x)}$. In fact, it follows from the proof of this proposition that the map $\Max{\F(\x)}\to\int^{\x}$ sending a homomorphism $h\colon \F(\x)\to\int$ to $h\circ\quot_{\restriction \x}\colon \x\to\int$ is a homeomorphism.

A particular class of implicit definitions is obtained as follows. Fix a set $\x$, a closed subset $X\subseteq \int^{\x}$, and a continuous function $\f\colon X\to \int$. If $y$ is a variable not in $\x$, then the graph of $\f$ can be identified with a closed subset of $\int^{\x,y}\cong\Max{\F(\x,y)}$. Hence, by Lemma~\ref{l:closed-sets-spectra}\ref{closed-Max}, the graph of $\f$ is homeomorphic to $\V(J_{\f})$ for some ideal $J_{\f}\subseteq \F(\x,y)$. Define the theory 
\begin{equation}\label{eq:sigma-f}
\Si_{\f}(\x,y)\coloneqq \{s(\x,y)\approx 0\mid s\in \quot^{-1}(J_{\f})\},
\end{equation}
i.e.\ $\Si_{\f}\coloneqq \Th{J_{\f}}$ with the notation in~\eqref{eq:Si[J]}. Note that an assignment $g\colon \x,y\to \int$ satisfies $\int,g\models \Si_{\f}$ if, and only if, it lies on the graph of $\f$ when regarded as a point of $\int^{\x,y}$. Thus, $\Si_{\f}$ implicitly defines $y$ over $\x$ because the graph of $\f$ is a functional relation:
\begin{lemma}\label{l:Sigma-f-implicit}
The theory $\Si_{\f}$ implicitly defines $y$ over $\x$.
\end{lemma}
\begin{proof}
Consider an assignment $f\colon \x\to \A$ with $\A\in\Delta$, and assume that $g\colon \x,y\to \A$ is an assignment extending $f$ and satisfying $\A,g\models \Si_{\f}$. We show that $g$ is the only assignment of the variables $\x,y$ with these properties. 

If $\A$ is the trivial algebra, then this is clearly true. Hence, let us suppose that $\A$ is non-trivial. We can assume without loss of generality that $\A=\int$. 
If $g'\colon \x,y\to \int$ is another assignment extending $f$ and satisfying $\int,g'\models \Si_{\f}$, and $\pi\colon \int^{\x,y}\epi\int^{\x}$ is the projection map, we get $\pi(g)=\pi(g')$ because both $g$ and $g'$ extend $f$. Therefore, since $g$ and $g'$ belong to the graph of $\f$, which is a functional relation, it must be $g=g'$.
\end{proof}
By definition, the theory $\Si_{\f}(\x,y)$ explicitly defines $y$ over $\x$ if there exists a term $s_y(\x)$ such that $\Si_{\f}\Log y\approx s_y$. 
By equation~\eqref{eq:int-generates}, this is equivalent to saying that, for every assignment $g\colon \x,y\to \int$, 
\[
\int, g\models \Si_{\f} \  \text{ entails } \ \int,g\models y\approx s_y. 
\]
We already observed that $\int, g\models \Si_{\f}$ if, and only if, $g$ belongs to the graph of $\f$. In turn, if $h\colon \F(\x,y)\to \int$ is the unique homomorphism extending $g$, $\pi\colon \int^{\x,y}\epi\int^{\x}$ is the projection on the $\x$-coordinates, and $\pi_y\colon \int^{\x,y}\epi \int$ is the projection on the $y$-coordinate,
\begin{align*}
\int,g\models y\approx s_y \ &\Longleftrightarrow \ h(\quot(y))=h(\quot(s_y)) \\ 
& \Longleftrightarrow \ \w{\quot(y)}(g)=\w{\quot(s_y)}(\pi(g)) \\
& \Longleftrightarrow \ \pi_y(g)=\w{\quot(s_y)}(\pi(g))
\end{align*}
because $\w{\quot(y)}\colon \int^{\x,y}\to \int$ coincides with $\pi_y$. Thus, 
\begin{equation}\label{eq:Si-f-explicit-def}
\Si_{\f} \text{ explicitly defines $y$ over $\x$ } \Longleftrightarrow \  \exists\phi\in\F(\x) \text{ such that } \w{\phi}_{\restriction X}=\f.
\end{equation}
Note that, by the previous discussion, for the left-to-right direction we can take $\phi\coloneqq \quot(s_y)$. Just observe that, for all assignments $g\in \int^{\x,y}$ lying on the graph of $\f$, $\pi_y(g)=\f(\pi(g))$. Hence, 
\[
\pi_y(g)=\w{\quot(s_y)}(\pi(g)) \text{ for all assignments $g$ with $\int, g\models \Si_{\f}$} \ \Longleftrightarrow \ \f(w)=\w{\quot(s_y)}(w) \ \forall w\in X.
\]
\begin{remark}
By considering all theories of the form $\Si_{\f}$, the right-hand condition in~\eqref{eq:Si-f-explicit-def} implies the following form of the Tietze-Urysohn Extension Theorem: Every continuous function $\f\colon X\to \int$ defined on a closed subset $X$ of a Tychonoff cube $\int^{\x}$ can be extended to a continuous function on $\int^{\x}$.
\end{remark}

In Theorem~\ref{t:equivalent-formulations} below, we will see that the existence of explicit definitions of the type~\eqref{eq:Si-f-explicit-def} is enough to deduce that all implicit definitions can be made explicit, i.e.\ that $\Log$ has the Beth definability property. We start by proving the following useful fact:

\begin{lemma}\label{l:eta-is-epi}
For any $\delta$-algebra $\A$, $\eta_{\A}\colon \A\to\C(\Max{\A},\int)$ is an epimorphism in $\Delta$.
\end{lemma}
\begin{proof}
Consider distinct homomorphisms $h_1,h_2\colon \C(\Max{\A},\int)\rightrightarrows \B$, for some $\B$ in $\Delta$. We must prove that $h_1\circ \eta_{\A}\neq h_2\circ \eta_{\A}$. The map $\eta_{\B}$ is injective by Theorem~\ref{t:semisimple-and-full}\ref{Delta-unit}, hence the latter condition is equivalent to $\eta_{\B}\circ h_1\circ \eta_{\A}\neq \eta_{\B}\circ h_2\circ\eta_{\A}$.
\[\begin{tikzcd}[column sep=3em]
\A \arrow{r}{\eta_{\A}} & \C(\Max{\A},\int) \arrow[yshift=4pt]{r}{\eta_{\B}\circ h_1} \arrow[yshift=-4pt]{r}[swap]{\eta_{\B}\circ h_2} & \C(\Max{\B},\int)
\end{tikzcd}\] 
By Theorem~\ref{t:Max-C-adj}\ref{counit}, the functor $\C\colon \KH^{\op}\to \MV$ is full, so there exist continuous functions $\f_1,\f_2\colon \Max{\B}\rightrightarrows \Max{\A}$ satisfying $\C{\f_1}=\eta_{\B}\circ h_1$ and $\C{\f_2}=\eta_{\B}\circ h_2$. Since $\eta_{\B}\circ h_1\neq \eta_{\B}\circ h_2$, there exists $\h\in\C(\Max{\A},\int)$ such that \[\h\circ \f_1=\C{\f_1}(\h)\neq\C{\f_2}(\h)=\h\circ \f_2.\] Let $x\in \Max{\B}$ be such that $\h(\f_1(x))\neq \h(\f_2(x))$. It is enough to find $a\in \A$ satisfying $\w{a}(\f_1(x))=\h(\f_1(x))$ and $\w{a}(\f_2(x))=\h(\f_2(x))$, for then we have 
\begin{equation*}
((\eta_{\B}\circ h_1\circ\eta_{\A})(a))(x)=(\C{\f_1}(\w{a}))(x)=\w{a}(\f_1(x))=\h(\f_1(x))
\end{equation*}
and similarly $((\eta_{\B}\circ h_2\circ\eta_{\A})(a))(x)=\h(\f_2(x))$, showing that $\eta_{\B}\circ h_1\circ \eta_{\A}\neq \eta_{\B}\circ h_2\circ\eta_{\A}$. The existence of such an $a\in \A$ follows from Corollary~\ref{c:interpolation-delta} by setting $C_1\coloneqq \{\f_1(x)\}$, $C_2\coloneqq \{\f_2(x)\}$, $r_1\coloneqq \h(\f_1(x))$, and $r_2\coloneqq \h(\f_2(x))$, and using the fact that $\f_1(x)\neq \f_2(x)$.
\end{proof}

\begin{theorem}\label{t:equivalent-formulations}
The following statements are equivalent:
\begin{enumerate}[label=\textup{(\arabic*)}]
\item $\Log$ has the Beth definability property.
\item For any set $\x$ and continuous function $\f\colon X\to\int$ defined on a closed subset $X\subseteq \int^{\x}$, the theory $\Si_{\f}(\x,y)$ from equation~\eqref{eq:sigma-f} explicitly defines $y$ in terms of $\x$.
\item\label{eta-iso} For every $\A\in\Delta$, the homomorphism $\eta_{\A}\colon \A\to \C(\Max{\A},\int)$ is an isomorphism. 
\item All epimorphisms in $\Delta$ are surjective.
\end{enumerate}
\end{theorem}
\begin{proof}
$1\Rightarrow 2$. This is an immediate consequence of Lemma~\ref{l:Sigma-f-implicit}.

$2\Rightarrow 3$. Since $\eta_{\A}\colon \A\to \C(\Max{\A},\int)$ is injective by Theorem~\ref{t:semisimple-and-full}\ref{Delta-unit}, it suffices to show that it is surjective. Consider an arbitrary continuous function $\f\colon \Max{\A}\to\int$ and let $\x$ be a set such that there exists a surjective homomorphism $p\colon \F(\x)\epi \A$. By Remark~\ref{rem:epis-to-monos}, the space $\Max{\A}$ can be identified with a closed subspace of $\Max{\F(\x)}\cong \int^{\x}$. If the theory $\Si_{\f}(\x,y)$ explicitly defines $y$ in terms of $\x$ then, in view of equation~\eqref{eq:Si-f-explicit-def}, there exists $\phi\in\F(\x)$ such that $\w{\phi}_{\restriction \Max{\A}}=\f$. But 
\[
\w{\phi}_{\restriction \Max{\A}}=\w{p(\phi)}=\eta_{\A}(p(\phi)),
\] 
and so $\eta_{\A}$ is surjective.

$3\Rightarrow 4$. It is enough to show that every homomorphism in $\Delta$ that is both an epimorphism and a monomorphism is an isomorphism. For any homomorphism $h\colon \A\to \B$ in $\Delta$, the naturality of $\eta$ yields the following commutative diagram:
\[\begin{tikzcd}
\A \arrow{d}[swap]{h} \arrow{r}{\eta_{\A}} & \C(\Max{\A},\int) \arrow{d}{\C(\Max{h})} \\
 \B \arrow{r}{\eta_{\B}} & \C(\Max{\B},\int)
\end{tikzcd}\]
If $h$ is epi, then $\Max{h}\colon \Max{\B}\to\Max{\A}$ is injective by Remark~\ref{rem:epis-to-monos}. We claim that $\Max{h}$ is surjective provided $h$ is mono, i.e.\ injective. 

Suppose $h$ is mono and identify $\A$ with a subalgebra of $\B$. By the congruence extension property for MV-algebras, see e.g.~\cite[Proposition~8.2]{GM2005}, for any homomorphism $k\colon \A\to \int$ the maximal (hence, proper) ideal $k^{-1}(0)$ of $\A$ generates a proper ideal of $\B$. The latter can then be extended to a maximal ideal $\m$ of $\B$ by Zorn's Lemma, and the unique homomorphism $h_{\m}\colon \B\to\int$ provided by H\"older's Theorem extends $k$. This shows that $\Max{h}$ is surjective whenever $h$ is mono. 

Therefore, if $h$ is both epi and mono, $\Max{h}$ is a continuous bijection between compact Hausdorff spaces, hence a homeomorphism. We conclude that $\C(\Max{h})$ is an isomorphism in $\Delta$. Since the square above commutes and $\eta_{\A},\eta_{\B}$ are isomorphisms, $h$ is also an isomorphism.

$4\Rightarrow 1$. The following argument, essentially due to Makkai~\cite[\S 1]{Makkai1993}, exploits the fact that the category $\KH$ is regular.\footnote{A category is \emph{regular} if (i) it has finite limits, (ii) every morphism factors as a regular epi followed by a mono, and (iii) regular epis are stable under pullbacks, cf.~\cite{BGvO71} or~\cite{Borceux2}. In $\KH$, the (regular epi, mono) factorisation of a continuous map is the usual factorisation through its set-theoretic image endowed with the subspace topology.} Suppose that a theory $\Si(\x,y)$, with $y$ a variable not in $\x$, implicitly defines $y$ over $\x$. Let $z$ be a variable that is distinct from $y$ and not contained in $\x$, and consider the following diagram in $\Delta$ (we write e.g.\ $\F(\x,y)/{\Si(\x,y)}$ for the quotient of $\F(\x,y)$ with respect to the congruence generated by the image of $\Si(\x,y)$ under the homomorphism $\quot\times \quot\colon \T(\x,y)^2\to \F(\x,y)^2$)
\[\begin{tikzcd}[row sep=2.5em, column sep=3.0em]
\F(\x) \arrow[r,"g"] \arrow[d,swap,"g"] &
  \F(\x,y)/{\Si(\x,y)} \arrow[d,"h_1"] \\
  \F(\x,y)/{\Si(\x,y)} \arrow[r,"h_2"] &
  \F(\x,y,z)/{\Si(\x,y)\cup \Si(\x,z)}
\end{tikzcd}\]
where 
\begin{itemize}[leftmargin=1em]
\item $g$ is the composition of $\F(\x)\mono \F(\x,y)$ with the quotient $\F(\x,y)\epi\F(\x,y)/{\Si(\x,y)}$, 
\item $h_1$ is the composition of the inclusion  $\F(\x,y)/{\Si(\x,y)}\mono \F(\x,y,z)/{\Si(\x,y)}$ with the quotient map $\F(\x,y,z)/{\Si(\x,y)}\epi \F(\x,y,z)/{\Si(\x,y)\cup \Si(\x,z)}$,
\item $h_2$ is obtained by first applying the isomorphism $\F(\x,y)/{\Si(\x,y)}\to \F(\x,z)/{\Si(\x,z)}$ that replaces $y$ by $z$, then the inclusion $\F(\x,z)/{\Si(\x,z)}\mono \F(\x,y,z)/{\Si(\x,z)}$, and finally the quotient $\F(\x,y,z)/{\Si(\x,z)}\epi \F(\x,y,z)/{\Si(\x,y)\cup \Si(\x,z)}$.
\end{itemize}
It is not difficult to see that the diagram above is a pushout square in $\Delta$. Now, consider the following equaliser diagram in the category $\Delta$:
\[\begin{tikzcd}
\B \arrow[hookrightarrow]{r}{i} & \F(\x,y)/{\Si(\x,y)} \arrow[yshift=4pt]{r}{h_1} \arrow[yshift=-4pt]{r}[swap]{h_2} & \F(\x,y,z)/{\Si(\x,y)\cup \Si(\x,z)}
\end{tikzcd}\]
Since $h_1\circ g=h_2\circ g$, by the universal property of $\B$ there is a homomorphism $j\colon \F(\x)\to \B$ such that $g=i\circ j$.
\begin{claim*}
$(j,i)$ is the (epi, regular mono) factorisation of $g\colon \F(\x)\to \F(\x,y)/{\Si(\x,y)}$.
\end{claim*} 
\begin{proof}[Proof of Claim]
If all epimorphisms in $\Delta$ are surjective then, for every $\A\in\Delta$, the embedding $\eta_{\A}\colon \A\to\C(\Max{\A},\int)$ is an isomorphism by Lemma~\ref{l:eta-is-epi}. Hence, the dual adjunction $\Max\dashv\C\colon\KH^{\op}\to\Delta$ yields an equivalence $\Delta\cong\KH^\op$. Let $\f\colon X\to Y$ be the continuous map in $\KH$ dual to $g$. Since $\KH$ is a regular category, the (regular epi, mono) factorisation of $\f$ is $(e,m)$, where $e\colon X\epi Z$ is the coequaliser of the kernel pair of $\f$, and $m\colon Z\mono Y$ is the unique morphism provided by the universal property of $Z$. See e.g.~\cite[p.~7]{BGvO71} for a proof. Recall that the kernel pair of $\f$ is obtained by taking the pullback of $\f$ along itself. Thus, by construction, the dual of $e$ is $i$ and the dual of $m$ is $j$. We then see that $(j,i)$ is the (epi, regular mono) factorisation of $g$.
\end{proof}
Since all epimorphisms in $\Delta$ are surjective, $j$ must be a surjection.
Because $\Si(\x,y)$ implicitly defines $y$ in terms of $\x$, we have 
\[
\Si(\x,y)\cup\Si(\x,z)\Log y\approx z.
\] 
Thus, the homomorphisms $h_1$ and $h_2$ coincide on the equivalence class of $y$ in $\F(\x,y)/{\Si(\x,y)}$. Let $\psi$ denote this equivalence class. Then $\psi\in \B$ and, by surjectivity of $j$, there is $\phi\in\F(\x)$ such that $j(\phi)=\psi$. If $s_y$ is any element of $\T(\x)$ whose image under $\quot\colon \T(\x)\epi\F(\x)$ is $\phi$, we get $g(\quot(s_y))=i(j(\phi))=\psi$. Hence $\Si(\x,y)\Log y\approx s_y$, as was to be proved.
\end{proof}

\begin{remark}
The equivalence between items 1 and 4 in Theorem~\ref{t:equivalent-formulations} is known, in the framework of abstract algebraic logic, as the Blok-Hoogland Theorem~\cite{BH2006}. In the particular case of the equational consequence $\vDash_{\mathcal{V}}$ associated with a Birkhoff variety $\mathcal{V}$, the Blok-Hoogland Theorem states that all epimorphisms in $\mathcal{V}$ are surjective if, and only if, $\vDash_{\mathcal{V}}$ has the so-called \emph{infinite Beth property}. This result does not apply in our setting because $\Delta$ is not a Birkhoff variety of algebras. However, a lengthy but rather straightforward verification shows that the Blok-Hoogland Theorem can be generalised to all varieties of infinitary algebras in the sense of S{\l}omi{\'n}ski~\cite{Slominski59}. Here we have opted for a more direct proof, specific to the variety $\Delta$, which emphasises the role of the theories $\Si_{\f}$.
\end{remark}

We are now in a position to prove the following result.
\begin{theorem}\label{th:Beth-prop}
The logic $\Log$ has the Beth definability property.
\end{theorem}
\begin{proof}
In view of Theorem~\ref{t:equivalent-formulations}, it suffices to show that for any set $\x$, closed subset $X\subseteq \int^{\x}$, and continuous function $\f\colon X\to\int$, the theory $\Si_{\f}(\x,y)$ explicitly defines $y$ in terms of $\x$. 

By Lemma~\ref{l:closed-sets-spectra}\ref{closed-Max}, there is an ideal $J$ of $\F(\x)$ such that $X\cong \V(J)$.  Let $\A\coloneqq \F(\x)/J$, with quotient map $p\colon \F(\x)\epi \A$. We have $\Max{\A}\cong X$, and so $\A$ can be identified with a subalgebra of $\C(X,\int)$ by composing the embedding $\eta_{\A}\colon \A\mono \C(\Max{\A},\int)$ with the isomorphism $\C(\Max{\A},\int)\cong \C(X,\int)$.
In view of Corollary~\ref{c:interpolation-delta}, for any two distinct assignments $u,v\colon \x\to \int$ that belong to $X$, there exists $a_{u,v}\in \A$ such that 
\[
\w{a_{u,v}}(u)=\f(u) \ \text{ and } \ \w{a_{u,v}}(v)=\f(v).
\] 
Let $s_{u,v}\in\T(\x)$ be a term whose image under the composite $p\circ\quot\colon\T(\x)\to \A$ is $a_{u,v}$. Fix an arbitrary $\epsilon\in (0,1]$ and define the theory
\[
\Gamma_{u,v}(\x,y)\coloneqq\Si_{\f}(\x,y)\cup \{(s_{u,v}\ominus y)\wedge \u{\epsilon}\approx \u{\epsilon}\},
\]
where $\ominus$ is truncated subtraction (see Section~\ref{s:l-groups-MV}) and $\u{\epsilon}$ is the definable constant corresponding to $\epsilon$, cf.~\eqref{eq:real-scalars}. We claim that, for every $u\in X$, the theory $\bigcup_{v\in X}{\Gamma_{u,v}}$ is unsatisfiable.

Assume towards a contradiction that $\bigcup_{v\in X}{\Gamma_{u,v}}$ is satisfiable. Then, by~\eqref{eq:int-generates}, there is an assignment $f\colon \x,y\to \int$ satisfying $\int,f\models  \Gamma_{u,v}$ for every $v\in X$. Note that the restriction $v'\colon \x\to \int$ of $f$ to $\x$ belongs to $X$ because $\int,f\models \Si_{\f}$. Thus, 
\begin{align*}
\int, f\models \Gamma_{u,v'} \ &\Longrightarrow \ \int,f\models \Si_{\f}(\x,y) \cup \{(s_{u,v'}\ominus y)\wedge \u{\epsilon}\approx \u{\epsilon}\} \\
&\Longrightarrow \ \int,f\models \{y\approx \u{\f(v')}\} \cup \{(s_{u,v'}\ominus y)\wedge \u{\epsilon}\approx \u{\epsilon}\} \\
&\Longrightarrow \ \int,f\models (s_{u,v'}\ominus \u{\f(v')})\wedge \u{\epsilon}\approx \u{\epsilon},
\end{align*}
contradicting the fact that $\w{a_{u,v'}}(v')=\f(v')$, i.e.\ $\w{a_{u,v'}}(v')\ominus \f(v')=0$. 

By compactness of $\Log$ (Lemma~\ref{l:compactness-logic-delta}), there exist $v_1,\ldots, v_m\in X$ such that $\bigcup_{i=1}^m{\Gamma_{u,v_i}}$ is unsatisfiable. That is, for each $w\in X$ there is $i\in\{1,\ldots,m\}$ such that $\w{a_{u,v_i}}(w)\ominus \f(w)<\epsilon$, and so $\w{a_{u,v_i}}(w)<\f(w)+\epsilon$. Therefore, the term $\lambda_u\coloneqq s_{u,v_1}\wedge \cdots\wedge s_{u,v_m}$ satisfies
\[
\forall w\in X, \ \ \w{\quot(\lambda_u)}(w)<\f(w)+\epsilon.
\]
Now, for any $u\in X$, define the theory
\[
\Gamma'_{u}(\x,y)\coloneqq\Si_{\f}(\x,y)\cup \{(y\ominus \lambda_u)\wedge \u{\epsilon}\approx \u{\epsilon}\}.
\]
Reasoning as before, it is not difficult to see that the theory $\bigcup_{u\in X}{\Gamma'_{u}}$ is unsatisfiable.
By compactness of $\Log$, there are $u_1,\ldots,u_n\in X$ such that, for each $w\in X$, $\f(w)\ominus \w{\quot(\lambda_{u_j})}(w)<\epsilon$ for some $j\in\{1,\ldots,n\}$, hence $\w{\quot(\lambda_{u_j})}(w)>\f(w)-\epsilon$. The term $\mu\coloneqq \lambda_{u_1}\vee\cdots \vee \lambda_{u_n}$ satisfies
\begin{equation}\label{eq:unif-appr}
\forall w\in X, \ \ \f(w)-\epsilon<\w{\quot(\mu)}(w)<\f(w)+\epsilon.
\end{equation}
Since $\epsilon\in (0,1]$ is arbitrary, equation~\eqref{eq:unif-appr} entails that $\f$ belongs to the closure of $\A$ in the topology induced by the uniform metric of $\C(X,\int)$. We claim that $\f\in \A$. (The following argument is already implicit in~\cite{MR2017}; we briefly recall it for the sake of completeness.)

Suppose that $\f$ is the uniform limit of a sequence $(\f_i)_{i\in \omega}\in \A^\omega$. We can assume without loss of generality that this sequence is increasing, cf.\ the proof of \cite[Lemma~7.5]{MR2017}. Since multiplication by any real number in $\int$ is definable in the language of $\delta$-algebras, we see that $(\frac{\f_i}{2})_{i\in\omega}\in \A^\omega$. Extract a subsequence $(\g_i)_{i\in \omega}$ of $(\frac{\f_i}{2})_{i\in\omega}$ satisfying, for every $i\in\omega$, \[\sup_{u'\in X}{\{|\g_i(u')-\g_{i-1}(u')|\}}\leq\frac{1}{2^i}.\]
Then an elementary computation shows that $\frac{\f}{2}=\delta(2\g_1,2^2(\g_2\ominus \g_1),\ldots, 2^i(\g_i\ominus \g_{i-1}),\ldots)\in \A$. For a proof, see \cite[Lemma~7.6]{MR2017}. We conclude that $\f=\frac{\f}{2}\oplus \frac{\f}{2}\in \A$.

To settle the theorem, pick $\phi\in \F(\x)$ such that $p(\phi)=\f$, where $p\colon \F(\x)\epi \A$ is the quotient map. Then $\w{\phi}_{\restriction X}=\f$. By equation~\eqref{eq:Si-f-explicit-def}, $\Si_{\f}$ explicitly defines $y$ over $\x$.
\end{proof}

To conclude this section, we show how to derive the Stone-Weierstrass Theorem for compact Hausdorff spaces from the Beth definability property of~$\Log$.

\begin{proof}[Proof of Theorem~\ref{t:SW-lattice-version}]\label{proof:SW}
Suppose $X$ is a non-empty compact Hausdorff space. If $G\subseteq \C(X,\R)$ satisfies the assumptions in the statement of Theorem~\ref{t:SW-lattice-version}, then $G$ is a unital $\ell$-subgroup of $\C(X,\R)$ that is divisible (as an Abelian group) and separates the points of $X$. Write $\o{G}$ for the closure of $G$ in the topology induced by the uniform metric and observe that $\o{G}$ is also a divisible unital $\ell$-subgroup of $\C(X,\R)$. We must prove that $\o{G}=\C(X,\R)$. By Theorem~\ref{t:gamma}, it suffices to show that the inclusion $\Gamma(\o{G})\mono \C(X,\int)$ is surjective, hence an isomorphism of MV-algebras. 
We claim that $\Gamma(\o{G})$ is a $\delta$-algebra that separates the points of $X$.

The Abelian $\ell$-group $G$ separates the points of $X$ and is divisible, hence its unit interval separates the points of $X$. A fortiori, the MV-subalgebra $\Gamma(\o{G})$ of $\C(X,\int)$ separates the points of $X$. Next, we show that $\Gamma(\o{G})$ is a $\delta$-subalgebra of $\C(X,\int)$, i.e.\ it is closed under the interpretation of the operation $\delta$ in $\C(X,\int)$.
Consider a sequence $\vec{\f}=(\f_i)_{i\in \omega}\in \Gamma(\o{G})^{\omega}$. In the algebra $\C(X,\int)$, we have 
\[ 
\delta(\vec{\f})=\sum_{i=1}^{\infty}{\dfrac{\f_i}{2^i}}=\lim_{n\to\infty}{\sum_{i=1}^n{\dfrac{\f_i}{2^i}}}
\] 
where the latter limit is uniform.
Since $\o{G}$ is divisible and closed under uniform limits, $\delta(\vec{\f})\in \o{G}$. Because each $\f_i$ belongs to the unit interval of $\o{G}$, so does $\delta(\vec{\f})$. 

To improve readability, write $\A$ for the $\delta$-algebra $\Gamma(\o{G})$. Since $\A$ separates the points of $X$ we have $\Max{\A}\cong X$ (cf.~\cite[Theorem~4.16]{M11}), so the inclusion $\A\mono \C(X,\int)$
can be obtained by composing the embedding $\eta_{\A}\colon\A\mono \C(\Max{\A},\int)$ with the isomorphism $\C(\Max{\A},\int)\cong \C(X,\int)$. By Theorem~\ref{th:Beth-prop}, the logic $\Log$ has the Beth definability property. It follows by Theorem~\ref{t:equivalent-formulations} that $\eta_{\A}$ is an isomorphism, and so the inclusion $\A\mono \C(X,\int)$ is surjective.
\end{proof}
\begin{remark}\label{rem:SW-implies-Beth}
The previous proof exploits the Beth definability property of $\Log$ to derive the Stone-Weierstrass Theorem. In turn, Theorem 8.1 in~\cite{MR2017} shows that an application of an appropriate version of the Stone-Weierstrass Theorem yields item~\ref{eta-iso} in Theorem~\ref{t:equivalent-formulations}, and thus also the Beth definability property of $\Log$. In this sense, the Beth definability property of $\Log$ is equivalent to the Stone-Weierstrass Theorem for compact Hausdorff spaces.
\end{remark}
%

\section{A Hilbert-style calculus for $\Log$}\label{s:Hilbert-calculus}
In this final section we introduce an infinitary propositional logic $\Lo$ by means of a Hilbert-style calculus and show that $\Lo$ is strongly complete with respect to a natural $\int$-valued semantics (Theorem~\ref{th:strong-comp} below). Up to a translation between terms and equations, the semantic notion of consequence associated with $\Lo$ coincides with the equational consequence relation $\Log$ defined in Section~\ref{s:logic-delta}. These results thus substantiate the logical nature of $\Log$.
 
Let us fix a countably infinite set of propositional variables $\var$ and consider the propositional connectives $\delta,\to$, and $\neg$. The set $\Form$ of formulas is defined inductively as follows:
\begin{enumerate}[label=(\roman*)]
\item Each propositional variable in $\var$ is a formula.
\item If $\alpha, \beta$ are formulas, then so are $\alpha\to\beta$ and $\neg \alpha$.
\item If $\langle \alpha_i\rangle$ is a countably infinite sequence of formulas, then $\delta(\langle \alpha_i\rangle)$ is a formula.
\end{enumerate}
We shall always assume that a countably infinite sequence $\langle \alpha_i\rangle$ is indexed by $\omega$, i.e.\ $\langle \alpha_i\rangle=\alpha_1,\alpha_2,\ldots,\alpha_i,\ldots$ for $i\in\omega$.

Note that $\Form$ can be regarded as an $\Lan_{\Delta}$-algebra. For any formulas $\alpha,\beta$ and any countably infinite sequence of formulas $\vec{\gamma}= \langle \gamma_i\rangle$, we set $\alpha\oplus\beta\coloneqq \neg \alpha\to \beta$ and $\delta(\vec{\gamma})\coloneqq \delta(\langle \gamma_i\rangle)$. The operation $\neg$ is defined in the obvious way, and $0\coloneqq\neg(\alpha\to\alpha)$ for an arbitrary formula~$\alpha$. (There is no canonical choice for the interpretation of $0$ in $\Form$. However, we will see that in the Lindenbaum-Tarski algebra obtained as an appropriate quotient of $\Form$, any formula of the form $\neg(\beta\to\beta)$ will belong to the equivalence class of $\neg(\alpha\to\alpha)$.)

Recall that any MV-algebra admits a derived connective $\to$ given by $x\to y\coloneqq\neg {x\oplus y}$. The connectives $\oplus$ and $\to$ are interdefinable as $x\oplus y=\neg x\to y$. Let us say that a function $f\colon \Form \to\int$ is a  \emph{valuation} if, for all sequences of formulas $\langle \alpha_i\rangle$ and all formulas $\alpha,\beta$,
\[
f(\delta(\langle \alpha_i\rangle))=\delta(\vec{f(\alpha_i)}), \ \ f(\alpha\to \beta)=f(\alpha)\to f(\beta),\ \text{ and } \ f(\neg \alpha)=\neg f(\alpha).
\] 
Note that a valuation is the same thing as an $\Lan_{\Delta}$-homomorphism $\Form\to\int$. A valuation $f\colon \Form\to\int$ \emph{satisfies} a formula $\alpha$ if $f(\alpha)=1$. For any set of formulas $\Theta\cup\{\alpha\}$, we say that $\alpha$ is a \emph{semantic consequence} of $\Theta$ provided that, for all valuations $f$, if $f$ satisfies all formulas in $\Theta$ then it also satisfies $\alpha$. The set of semantic consequences of $\Theta$ is denoted by~$\Theta^{\vDash}$. A formula is a \emph{tautology} if it belongs to $\emptyset^{\vDash}$.

Every formula $\alpha$ yields an $\Lan_{\Delta}$-term $t_{\alpha}$ in the same variables obtained by replacing each occurrence of $\alpha\to \beta$ by $\neg \alpha\oplus \beta$. Conversely, replacing each occurrence of $\alpha\oplus\beta$ and $0$ by $\neg\alpha\to\beta$ and $\neg(\alpha\to\alpha)$, respectively, we can associate a formula $\alpha_t$ with any $\Lan_{\Delta}$-term $t$. The following fact is an immediate consequence of equation~\eqref{eq:int-generates}.
\begin{lemma}\label{l:translation-formulas-terms}
The following hold for any $\Theta\cup\{\alpha\}\subseteq \Form$ and set of $\Lan_{\Delta}$-terms $S\cup\{t\}$: 
\begin{enumerate}[label=\textup{(\alph*)}]
\item\label{form-to-term} $\alpha\in\Theta^{\vDash}$ if, and only if, $\{t_{\beta}\approx 1\mid \beta\in\Theta\}\Log t_{\alpha}\approx 1$;
\item\label{term-to-form} $\{s\approx 1\mid s\in S\}\Log t\approx 1$ if, and only if, $\alpha_{t}\in\{\alpha_s\mid s\in S\}^{\vDash}$. 
\end{enumerate}
\end{lemma}
Also, recall that in the variety $\Delta$ any equation is equivalent to one of the form $t\approx 1$. Therefore, upon identifying formulas with $\Lan_{\Delta}$-terms, we see that the notion of semantic consequence defined above coincides with the equational consequence relation $\Log$.

Next, we introduce a logic $\Lo$ by adding finitely many axiom schemata to the usual axiomatisation of {\L}ukasiewicz propositional logic $\Luk$ (see e.g.~\cite[\S 4.3]{cdm2000}). To start with, let us recall the axioms for $\Luk$ (in the language expanded with the infinitary connective $\delta$). For arbitrary formulas $\alpha,\beta\in\Form$, these are:
\begin{multicols}{2}
\begin{enumerate}[label=(\L\arabic*), itemsep=0ex]
\item\label{axsch-MV1} $\alpha\to (\beta\to \alpha)$
\item\label{axsch-MV2} $(\alpha\to\beta)\to ((\beta\to\gamma)\to(\alpha\to\gamma))$
\item\label{axsch-MV3} $((\alpha\to\beta)\to\beta)\to((\beta\to\alpha)\to\alpha)$
\item\label{axsch-MV4} $(\neg\alpha\to\neg\beta)\to(\beta\to\alpha)$
\end{enumerate}
\end{multicols}

If $\alpha$ is any formula, we denote by $\langle \alpha\rangle$ the countably infinite sequence of constant value $\alpha$. Further, we write $\alpha,\langle \beta_i\rangle$ for the sequence $\alpha, \beta_1,\beta_2,\ldots$.
Finally, for convenience of notation, given any formula $\alpha$ we write
\[
\ha \alpha\coloneqq \delta(\alpha, \langle \neg (\alpha\to\alpha)\rangle).
\]
The logic $\Lo$ is defined by the axiom schemata \ref{axsch-MV1}--\ref{axsch-MV4} together with the following axioms: 
\begin{multicols}{2}
\begin{enumerate}[label=($\Delta$\arabic*), itemsep=0ex]
\item\label{axsch-Delta4} $\neg(\delta(\langle \alpha_i\rangle)\to \ha \alpha_1)\leftrightarrow \ha\delta(\langle \alpha_i\rangle_{i>1})$
\item\label{axsch-Delta5} $\ha\delta(\langle\alpha_i\rangle)\leftrightarrow \delta(\langle\ha\alpha_i\rangle)$ 
\item\label{axsch-Delta6} $\delta(\langle\alpha\rangle)\leftrightarrow \alpha$ 
\item\label{axsch-Delta10} $\ha\delta(\langle \alpha_i\rangle) \leftrightarrow \delta(\neg(\alpha\to\alpha),\langle \alpha_i\rangle)$
\item\label{axsch-Delta3} $\delta(\langle\alpha_i\rangle)\to\delta(\langle\neg\alpha_i\to\beta_i\rangle)$
\item\label{axsch-Delta2} $\ha\neg(\alpha\to\beta)\leftrightarrow \neg(\ha\alpha\to \ha\beta)$ 
\item\label{axsch-Delta9} $\delta(\langle \alpha_i\to \beta_i\rangle)\to (\delta(\langle \alpha_i\rangle)\to\delta(\langle \beta_i\rangle))$ 
\end{enumerate}
\end{multicols}
\noindent where $\langle \alpha_i \rangle_{i>1}$ denotes the truncated sequence $\alpha_2,\alpha_3,\ldots$ and $\alpha\leftrightarrow \beta$ stands for ``$\alpha\to\beta$ and $\beta\to\alpha$''. It is not difficult to see that an axiom of the form $\alpha\leftrightarrow \beta$ could be replaced by $\neg((\alpha\to\beta)\to\neg(\beta\to\alpha))$ (however, this would result in unwieldy expressions).

If $\eta$ is an ordinal number, its \emph{successor} is $\eta+1\coloneqq \eta\cup \{\eta\}$. The cardinal associated with $\eta$ is denoted by $\card{\eta}$. Let $\Theta\cup\{\alpha\}$ be an arbitrary set of formulas. A \emph{proof} of $\alpha$ from $\Theta$ is a sequence of formulas $(\alpha_i)_{i\in\eta+1}\in \Form^{\eta+1}$ such that $\card{\eta}\leq \aleph_0$, $\alpha=\alpha_{\eta}$, and each member of $(\alpha_i)_{i\in\eta}$ is either an axiom, or an element of $\Theta$, or can be obtained from (some of) its predecessors using one of the following rules: 
\begin{equation*}
\begin{gathered}
\infer[(\text{Modus Ponens})]
{\beta}{\alpha \hspace{2em} \alpha\to\beta}
\hspace{4em}
\infer[(\text{$\delta$-rule})]
{\delta(\langle\alpha_i\rangle)}{\alpha_1 \hspace{0.6em} \alpha_2 \hspace{0.6em} \cdots \hspace{0.6em} \alpha_i \hspace{0.6em} \cdots \hspace{1em} (i\in\omega)}
\end{gathered}
\end{equation*}
If a formula $\alpha$ admits a proof from a set of formulas $\Theta$, then we say that $\alpha$ is a \emph{syntactic consequence} of $\Theta$. The set of syntactic consequences of $\Theta$ is denoted by $\Theta^{\vdash}$. If $\alpha\in\emptyset^{\vdash}$, we say that $\alpha$ is \emph{provable} and write $\Lo \alpha$. 
For instance, the same proofs as for {\L}ukasiewicz logic (see~\cite[Proposition~4.3.4]{cdm2000}) show that $\Lo \alpha\to\alpha$, $\Lo \alpha\to\neg\neg\alpha$, and $\Lo \neg\neg\alpha\to\alpha$ for any formula $\alpha$.

We aim to prove that $\Lo$ is strongly complete with respect to the $\int$-valued semantics of formulas given above. To this end, we exploit the well known construction of the Lindenbaum-Tarski algebra of a propositional logic.
Let us define a relation $\equiv$ on $\Form$ by setting, for all $\alpha,\beta\in\Form$, 
\[
\alpha\equiv\beta \ \text{ if, and only if, } \ (\, \Lo \alpha\to \beta \ \text{and} \ \Lo\beta\to\alpha \, ).
\]
For every formula $\alpha$, let $[\alpha]\coloneqq \{\beta\in \Form\mid \alpha\equiv \beta\}$.

\begin{proposition}
The following statements hold:
\begin{enumerate}[label=\textup{(\alph*)}]
\item The relation $\equiv$ is a congruence on the $\Lan_{\Delta}$-algebra $\Form$.
\item The quotient $\Form/{\equiv}=\{[\alpha]\mid \alpha\in\Form\}$ is a $\delta$-algebra.
\end{enumerate}
\end{proposition}
\begin{proof}
(a) Reasoning in the same way as for {\L}ukasiewicz logic, it is not difficult to see that $\equiv$ is an equivalence relation that is compatible with the MV-algebraic operations (cf.~\cite[Theorem~4.4.1]{cdm2000}). It remains to show that, given sequences of formulas $(\alpha_i)_{i\in\omega}$ and $(\beta_i)_{i\in\omega}$, if $\alpha_i\equiv \beta_i$ for every $i\in\omega$ then $\delta\langle \alpha_i\rangle \equiv \delta\langle \beta_i\rangle$. In turn, this follows from~\ref{axsch-Delta9} and the $\delta$-rule.

(b) Again, adapting the corresponding proof for {\L}ukasiewicz logic (cf.~\cite[Corollary~4.4.4]{cdm2000}), it is not difficult to see that $\Form/{\equiv}$ is an MV-algebra. So, it suffices to prove that $\Form/{\equiv}$ satisfies  equations~\ref{Delta1}--\ref{Delta6} in Definition~\ref{d:delta}.

For~\ref{Delta1}, we start by proving that $\delta([\alpha_1],\vec{0})\ominus \delta(\vec{[\alpha_i]})=0$, which will imply
\begin{equation}\label{eq:distance-halved}
\d(\delta(\vec{[\alpha_i]}),\delta([\alpha_1],\vec{0}))=\neg (\delta(\vec{[\alpha_i]})\to \delta([\alpha_1],\vec{0})).
\end{equation}
Note that $\delta([\alpha_1],\vec{0})\ominus \delta(\vec{[\alpha_i]})=0$ if, and only if, $\delta([\alpha_1],\vec{0})\to \delta(\vec{[\alpha_i]})=1$. We now make use of the following easy observation:
\begin{claim*} 
If $\beta$ is a provable formula, then $\Lo \neg \beta\to \alpha$ for any formula $\alpha$.
\end{claim*} 
\begin{proof}
Recall that $\Lo \beta\to \neg\neg\beta$. If $\Lo \beta$ then, by Modus Ponens, $\Lo \neg\neg\beta$. So, for any formula $\alpha$, we have $\Lo \neg\alpha\to \neg\neg\beta$ by~\ref{axsch-MV1} and therefore $\Lo\neg\beta\to\alpha$ by~\ref{axsch-MV4}.
\end{proof}
Let us fix an arbitrary provable formula $\beta$ (for instance, an axiom). Then $\Lo \alpha_1\to \alpha_1$ and, by the Claim, $\Lo \neg\beta\to \alpha_i$ for all $i>1$. An application of the $\delta$-rule yields
\[
\Lo\delta(\alpha_1\to\alpha_1,\langle \neg \beta\to\alpha_i\rangle_{i>1}).
\] 
By~\ref{axsch-Delta9} and Modus Ponens, we get $\Lo \delta(\alpha_1,\langle \neg\beta\rangle)\to\delta(\langle \alpha_i\rangle)$. Hence $\delta([\alpha_1],\vec{0})\to \delta(\vec{[\alpha_i]})=1$.
Thus, by equation~\eqref{eq:distance-halved},
\[
\d(\delta(\vec{[\alpha_i]}),\delta([\alpha_1],\vec{0}))=\neg(\delta(\vec{[\alpha_i]})\to \delta([\alpha_1],\vec{0}))= \neg(\delta(\vec{[\alpha_i]})\to \ha\alpha_1).
\]
In turn,~\ref{axsch-Delta4} and~\ref{axsch-Delta10} entail
\[
\neg(\delta(\vec{[\alpha_i]})\to \ha\alpha_1)= \ha\delta(\vec{[\alpha_i]}_{i>1})=\delta(0,\vec{[\alpha_i]}_{i>1}).
\]
This settles item~\ref{Delta1} in Definition~\ref{d:delta}. Items~\ref{Delta2},~\ref{Delta3},~\ref{Delta4},~\ref{Delta5} and~\ref{Delta6} follow at once from~\ref{axsch-Delta5}, \ref{axsch-Delta6}, \ref{axsch-Delta10}, \ref{axsch-Delta3} and \ref{axsch-Delta2}, respectively.
\end{proof}

We refer to the $\delta$-algebra $\Form/{\equiv}$ as the \emph{Lindenbaum-Tarski algebra} of the logic $\Lo$, and denote it by $\LT$. 
Note that, by~\ref{axsch-MV1}, for any provable formula $\alpha$ we have $\alpha\equiv\beta$ if, and only if, $\beta$ is provable. As
\[
1=\neg[\neg(\alpha\to \alpha)]=[\alpha\to\alpha]
\]
in the Lindenbaum-Tarski algebra $\LT$, we see that $[\beta]=1$ if, and only if, $\beta$ is provable.

\begin{theorem}[Completeness]\label{th:completeness}
A formula $\alpha\in\Form$ is provable if, and only if, it is a tautology. That is, $\emptyset^{\vdash}=\emptyset^{\vDash}$. 
\end{theorem}
\begin{proof}
As the axioms for $\Lo$ are easily seen to be tautologies, and Modus Ponens and the $\delta$-rule preserve tautologies, we have $\emptyset^{\vdash}\subseteq\emptyset^{\vDash}$. For the converse inclusion, assume that ${\alpha\notin \emptyset^{\vdash}}$. Then $[\alpha]\neq 1$ in the Lindenbaum-Tarski algebra $\LT$. Since any $\delta$-algebra embeds into a power of the $\delta$-algebra $\int$ by Theorem~\ref{t:semisimple-and-full}\ref{Delta-unit}, there exists a $\delta$-homomorphism $h\colon \LT\to \int$ such that $h([\alpha])\neq 1$. Composing $h$ with the canonical quotient $\Form\twoheadrightarrow \LT$, we obtain a valuation $\Form \to \int$ that does not satisfy $\alpha$. Therefore, $\alpha\notin \emptyset^{\vDash}$.
\end{proof}

Finally, we improve the previous result by showing that the syntactic and semantic notions of consequence coincide for arbitrary sets of premises. To this end, recall that a subspace $E$ of a topological space is \emph{Lindel\"of} if any open cover of $E$ admits a countable subcover.

\begin{theorem}[Strong Completeness]\label{th:strong-comp}
For any set of formulas $\Theta$, $\Theta^{\vdash}=\Theta^{\vDash}$. 
\end{theorem}
\begin{proof}
We only prove that $\Theta^{\vDash}\subseteq\Theta^{\vdash}$, as the other inclusion is readily seen to hold. 
We first settle the case where $\Theta$ is countable and then deduce the general one. It is convenient to assume that $\Theta=\{\beta_i\mid i\in\omega\}$ (if the original set of formulas is finite, it suffices to add countably many copies of an arbitrary provable formula, e.g.\ an axiom). If $\alpha\in\Theta^{\vDash}$, then 
\[
\{t_{\beta_i}\approx 1\mid i\in\omega\}\Log t_{\alpha}\approx 1
\]
by Lemma~\ref{l:translation-formulas-terms}\ref{form-to-term}. Let $\x$ be a set such that the variables of each $\beta_i$ are contained in $\x$ and note that, for any assignment $f\colon \x\to\int$,
\[
\int, f\models \{t_{\beta_i}\approx 1\mid i\in\omega\} \ \Longleftrightarrow \ \int, f\models \delta(\vec{t_{\beta_i}})\approx 1.
\] 
Just observe that, for all $\vec{r}\in\int^{\omega}$, $\sum_{i=1}^{\infty}\tfrac{r_i}{2^i}=1$ if, and only if, $\vec{r}$ is the constant sequence of value $1$. Thus $\{\delta(\vec{t_{\beta_i}})\approx 1\}\Log t_{\alpha}\approx 1$ and so, by Lemma~\ref{l:deduction}, there is a $k\in\N$ such that 
\[
\emptyset \Log (\delta(\vec{t_{\beta_i}})^k\to t_{\alpha})\approx 1.
\]
Let $\alpha_{t_{\alpha}}$ and $\gamma$ be the formulas associated, respectively, with the terms $t_{\alpha}$ and $\delta(\vec{t_{\beta_i}})^k$.
An application of Lemma~\ref{l:translation-formulas-terms}\ref{term-to-form} yields $\gamma\to \alpha_{t_{\alpha}}\in\emptyset^{\vDash}$ and so, by Theorem~\ref{th:completeness}, $\gamma\to \alpha_{t_{\alpha}}\in\emptyset^{\vdash}$. Upon observing that $\Lo \alpha_{t_{\alpha}}\to \alpha$, we get $\gamma\to \alpha\in\emptyset^{\vdash}$ by~\ref{axsch-MV2}. To conclude that $\alpha\in\Theta^{\vdash}$, it remains to prove that $\gamma\in\Theta^{\vdash}$.

If $\tilde{\beta}_i$ is the formula associated with $t_{\beta_i}$, then it is not difficult to see that $\Lo \beta_i\to \tilde{\beta}_i$, and so $\tilde{\beta}_i\in\Theta^{\vdash}$ for all $i\in\omega$. Hence, by the $\delta$-rule, also $\delta(\langle \tilde{\beta}_i\rangle)\in \Theta^{\vdash}$. Now, for any two terms $s,t$ corresponding to formulas $\alpha_s$ and $\alpha_t$, respectively, the term $s\odot t$ corresponds to the formula $\neg(\neg\neg\alpha_s\to\neg\alpha_t)$. We claim that $\neg(\neg\neg\alpha_s\to\neg\alpha_t)\in \Theta^{\vdash}$ whenever $\alpha_s,\alpha_t\in \Theta^{\vdash}$. An easy inductive argument then shows that $\gamma\in\Theta^{\vdash}$.

Suppose that $\alpha_s,\alpha_t\in \Theta^{\vdash}$ and note that, for any $\alpha,\beta,\gamma\in\Form$,
\begin{align}
&\Lo (\alpha\to (\beta\to\gamma))\to(\beta\to(\alpha\to\gamma)) \label{eq:der1}\\
&\Lo (\alpha\to\neg \beta)\to (\beta\to\neg\alpha). \label{eq:der2}
\end{align}
(The same proof as for {\L}ukasiewicz logic applies here, mutatis mutandis, cf.~\cite[Proposition~4.3.4]{cdm2000}.) Equation~\eqref{eq:der1}, together with $\Lo (\neg\neg\alpha\to\neg\beta)\to(\neg\neg\alpha\to\neg\beta)$, entails $\Lo \neg\neg\alpha\to ((\neg\neg\alpha\to\neg\beta)\to\neg\beta)$ and so, by~\eqref{eq:der2}, $\Lo \neg\neg\alpha\to (\beta\to \neg(\neg\neg\alpha\to\neg\beta))$. As $\Lo \alpha\to\neg\neg\alpha$, we get 
\[
\Lo \alpha\to (\beta\to \neg(\neg\neg\alpha\to\neg\beta)).
\]
Therefore, $\alpha_s,\alpha_t\in \Theta^{\vdash}$ entails $\neg(\neg\neg\alpha_s\to\neg\alpha_t)\in \Theta^{\vdash}$.

Now, for the general case, let $\Theta=\{\beta_i\mid i\in I\}$ be an arbitrary set of formulas. We claim that there exists a countable subset $\tilde{\Theta}\subseteq \Theta$ such that $\Theta^{\vDash}\subseteq\tilde{\Theta}^{\vDash}$. It then follows from the argument above that
\[
\Theta^{\vDash}\subseteq\tilde{\Theta}^{\vDash}\subseteq\tilde{\Theta}^{\vdash}\subseteq \Theta^{\vdash}.
\] 
Consider the set 
\[
S\coloneqq \V(\{\neg\quot(t_{\beta_i})\mid i\in I\})\subseteq \Max{\F(\x)}
\] 
where $\x\coloneqq \var$ is the countable set of propositional variables from which the formulas of $\Lo$ are built. By Proposition~\ref{p:spectrum-of-free}, $S$ can be identified with a closed subset of the Tychonoff cube $\int^{\x}$. Setting $S_i\coloneqq \V(\neg\quot(t_{\beta_i}))$ for each $i\in I$, we get $S=\bigcap_{i\in I}{S_i}$.
Recall that any second-countable space is Lindel\"of (see e.g.~\cite[Theorem~16.9]{Willard1970}) and any subspace of a second-countable space is also second-countable, so any subspace of a second-countable space is Lindel\"of. In particular, since $\int^{\x}$ is second-countable, the complement $S^{\complement}$ of $S$ is Lindel\"of. Hence there exists a countable subset $J\subseteq I$ such that $S^{\complement}\subseteq \bigcup_{i\in J}{S_i^{\complement}}$, and so
\begin{equation}\label{eq:countable-cover}
\bigcap_{i\in J}{S_i}\subseteq S.
\end{equation}
Note that the points of $S$ (respectively, of $\bigcap_{i\in J}{S_i}$) are in bijection with the assignments $f\colon \x\to \int$ that satisfy $\int, f\models t_{\beta_i}\approx 1$ for all $i\in I$ (respectively, for all $i\in J$). Thus, equation~\eqref{eq:countable-cover} entails that $\{t_{\beta_i}\approx 1\mid i\in J\}\Log \{t_{\beta_i}\approx 1\mid i\in I\}$. Upon setting $\tilde{\Theta}\coloneqq \{\beta_i\mid i\in J\}$, an application of Lemma~\ref{l:translation-formulas-terms}\ref{form-to-term} yields $\Theta^{\vDash}\subseteq\tilde{\Theta}^{\vDash}$.
\end{proof}

Let us remark that {\L}ukasiewicz logic $\Luk$ is complete---but \emph{not} strongly complete---with respect to its usual (Bolzano-Tarski) $\int$-valued semantics. This problem can be rectified by considering a different notion of model for sets of $\Luk$-formulas, giving rise to a \emph{differential semantics} of {\L}ukasiewicz logic \cite{Mundici2015}. This contrasts with Theorem~\ref{th:strong-comp}, which establishes the strong completeness of $\Lo$ with respect to its $\int$-valued semantics. 

\begin{remark}
Although we have defined formulas starting from a countably infinite set of propositional variables, the results in this section can be generalised to uncountable sets of variables. The proof of the Completeness theorem remains unchanged, as a single formula $\alpha$ can only use countably many variables. We briefly indicate how to adapt the proof of the Strong Completeness theorem. Let $\Theta=\{\beta_i\mid i\in I\}$ be any set of formulas with variables in a (possibly uncountable) set $\y$. For the interesting direction, we must prove that for any formula $\alpha$ with variables in $\y$, $\alpha\in\Theta^{\vDash}$ entails $\alpha\in \Theta^{\vdash}$. Suppose that $\alpha\in\Theta^{\vDash}$. It suffices to show that $\alpha\in\tilde{\Theta}^{\vDash}$ for a countable subset $\tilde{\Theta}\subseteq \Theta$, for then the set of variables appearing in either $\alpha$ or one of the formulas in $\tilde{\Theta}$ is countable, and so Theorem~\ref{th:strong-comp} yields $\alpha\in \tilde{\Theta}^{\vdash}\subseteq \Theta^{\vdash}$.

Denote by $\x$ the countable subset of $\y$ consisting of the variables appearing in $\alpha$, and let $\pi\colon \int^{\y}\twoheadrightarrow \int^{\x}$ be the projection map. Set $S_i\coloneqq \V(\neg\quot(t_{\beta_i}))\subseteq \int^{\y}$ for each $i\in I$, and $T\coloneqq \V(\neg\quot(t_{\alpha}))\subseteq \int^{\x}$. Reasoning as in the proof of the Strong Completeness theorem, we see that $\alpha\in\Theta^{\vDash}$ implies
\[
\bigcap_{i\in I}{S_i}\subseteq \pi^{-1}(T).
\]
Because $\int^{\x}$ is second-countable, $T^{\complement}$ is Lindel\"of. By \cite[Theorem~3.8.8]{engelking}, the preimage $\pi^{-1}(T^{\complement})=\pi^{-1}(T)^{\complement}$ is also Lindel\"of. Hence, there exists a countable subset $J\subseteq I$ such that $\pi^{-1}(T)^{\complement}\subseteq \bigcup_{i\in J}{S_i^{\complement}}$, and so
\[
\bigcap_{i\in J}{S_i}\subseteq \pi^{-1}(T).
\]
It follows easily that $\alpha\in\tilde{\Theta}^{\vDash}$ where $\tilde{\Theta}\coloneqq \{\beta_i\mid i\in J\}$, as desired.
\end{remark}

\section{Conclusion}
In~\cite[p.~467]{st2}, Stone observes that the proof of Weierstrass' Approximation Theorem can be divided into two parts. The first part, of which he provides a  generalisation, can be regarded, in his words, as ``algebraico-topological'', while he refers to the second part as the ``analytical kernel'' of the proof. On the other hand, Banaschewski showed in~\cite{Banaschewski2001} that even this analytical kernel can be proved by algebraic means in the setting of $f$-rings. In a sense, Banaschewski fully brought out the \emph{algebraic} content of the Stone-Weierstrass Theorem, thus concluding a process started by Stone himself.

In this paper, we have exposed the \emph{logical} content of the Stone-Weierstrass Theorem by showing that it can be ultimately regarded as an equivalent form of the Beth definability property of the logic $\Log$ (see Remark~\ref{rem:SW-implies-Beth}). 
In the same spirit, other properties of compact Hausdorff spaces admit a translation into properties of the logic $\Log$, and vice versa. For instance, the Joint Consistency Theorem for $\Log$ yields a Urysohn's Lemma for Tychonoff cubes (Remark~\ref{rem:Urysohn}), and it is not difficult to see that the fact that any continuous map between compact Hausdorff spaces is closed implies the deductive interpolation property of $\Log$ (in fact, even the right uniform deductive interpolation property~\cite{vGMT}). 

\begin{ack}
I am grateful to Nick Bezhanishvili for several useful comments on an earlier draft of this article, and to the anonymous referees for their valuable suggestions and for drawing my attention to the completeness problem for the logic $\Log$, which led to the results in Section~\ref{s:Hilbert-calculus}.
\end{ack}

\bibliographystyle{amsplain}

\providecommand{\MR}{\relax\ifhmode\unskip\space\fi MR }
\providecommand{\MRhref}[2]{%
  \href{http://www.ams.org/mathscinet-getitem?mr=#1}{#2}
}
\providecommand{\href}[2]{#2}

\end{document}